\documentclass[leqno,12pt]{article}

\usepackage{empheq}
\usepackage[margin=0.9 in  ]{geometry}         
\usepackage{longtable}
\usepackage{graphicx}
\usepackage{caption,subcaption}
\usepackage{stmaryrd}
\usepackage{framed}
\usepackage{amssymb}
\usepackage{epstopdf}
\usepackage{amsmath,amsfonts,amssymb}
\usepackage{enumerate}
\usepackage{multirow}
\usepackage{fixltx2e}
\usepackage{mathpazo}

\usepackage{enumitem}
\usepackage{xcolor}
\usepackage{soul}

\usepackage[toc,page]{appendix}

\usepackage{float}
\usepackage{theorem}
\usepackage{array}
\usepackage{booktabs}

\definecolor{myblue}{rgb}{.9, .9, 1}
  \newcommand*\mybluebox[1]{%
    \colorbox{myblue}{\hspace{1em}#1\hspace{1em}}}

\newcommand*\mywhitebox[1]{%
   \colorbox{white}{\hspace{1em}#1\hspace{1em}}}

\newcommand{\scal}[2]{\left\langle{#1},{#2}  \right\rangle}
\newcommand{\menge}[2]{\big\{{#1}~\big |~{#2}\big\}}

\newcommand{\kkk}{\ensuremath{{k\in{\mathbb N}}}}

\newcommand{\ran}{\ensuremath{\operatorname{ran}}}
\newcommand{\card}{\ensuremath{\operatorname{card}}}

\newtheorem{theorem}{Theorem}[section]
\newtheorem{lemma}[theorem]{Lemma}

\newtheorem{proposition}[theorem]{Proposition}
\newtheorem{definition}[theorem]{Definition}
\theoremstyle{plain}{\theorembodyfont{\rmfamily}
}
\theoremstyle{plain}{\theorembodyfont{\rmfamily}
}
\theoremstyle{plain}{\theorembodyfont{\rmfamily}
}
\theoremstyle{plain}{\theorembodyfont{\rmfamily}
\newtheorem{example}[theorem]{Example}}

\theoremstyle{plain}{\theorembodyfont{\rmfamily}
\newtheorem{remark}[theorem]{Remark}}
\def\proof{\noindent{\it Proof}. \ignorespaces}
\def\endproof{\ensuremath{\hfill \quad \blacksquare}}

\providecommand{\ve}{\varepsilon}

\providecommand{\bx}{\mathbf{x}}

\providecommand{\by}{\mathbf{y}}

\providecommand{\bX}{\mathbf{X}}

 \providecommand{\inte}{\operatorname{int}}
 \providecommand{\dom}{\operatorname{dom}}
 \providecommand{\sgn}{\operatorname{sgn}}
\providecommand{\conv}{\operatorname*{conv}}

  \providecommand{\prox}{\operatorname{P}}

\providecommand{\RR}{\mathbb{R}}
\providecommand{\RX}{\left]-\infty,+\infty\right]}
\providecommand{\RP}{\mathbb{R}_+}
\providecommand{\RM}{\mathbb{R}_-}
\providecommand{\RPP}{\mathbb{R}_{++}}
\providecommand{\RMM}{\mathbb{R}_{--}}

\providecommand{\NN}{\mathbb{N}}

\newcommand{\sepp}{\setlength{\itemsep}{-3pt}}

\newcommand{\argmin}{\ensuremath{\operatorname*{argmin}}}

\newcommand{\uball}{B}

\allowdisplaybreaks

\title{Stadium norm and Douglas--Rachford splitting:\\
a new approach to road design optimization}

\author{Heinz H.\ Bauschke\thanks{Mathematics, Irving K.~Barber School,
University of British Columbia,  Kelowna, B.C.~V1V~1V7, Canada.
Email: {\tt heinz.bauschke@ubc.ca}},  ~
Valentin R.\ Koch\thanks{Information Modeling \& Platform
Products Group (IPG), Autodesk, Inc.  Email: {\tt valentin.koch@autodesk.com}
}~~ and ~Hung M.\ Phan\thanks{Mathematics, Irving K.~Barber School,
University of British Columbia,  Kelowna, B.C.~V1V~1V7, Canada. 
Email: {\tt hung.phan@ubc.ca}}}

\date{September 29, 2014} 
\RequirePackage[normalem]{ulem} 
\RequirePackage{color} 

\makeatletter
\newcommand\appendix@section[1]{%
\refstepcounter{section}%
\orig@section*{Appendix \@Alph\c@section: #1}%
\addcontentsline{toc}{section}{Appendix \@Alph\c@section: #1}%
}
\let\orig@section\section
\g@addto@macro\appendix{\let\section\appendix@section}
\makeatother

\begin{document}

\maketitle

\begin{abstract}
\noindent
The basic optimization problem of road design is quite
challenging due to a objective function that is the sum of
nonsmooth functions and the presence of set constraints. 
In this paper, we model and solve this problem by employing the
Douglas--Rachford splitting algorithm. This requires a careful
study of new proximity operators related to minimizing area and
to the stadium norm. 
We compare our algorithm to a state-of-the-art 
projection algorithm. Our numerical results illustrate the
potential of this algorithm to significantly reduce cost in road design. 
\end{abstract}



\setlength{\parskip}{6pt}

\section{Introduction}

\subsection{The road design problem}

We set 
\begin{empheq}[box=\mybluebox]{equation}
X = \RR^n 
\end{empheq}
and write $x=(x_1,\ldots,x_n)$ for a vector in $X$.
Now fix
\begin{empheq}[box=\mybluebox]{equation}\label{e:t1-tn}
t=(t_1,\ldots,t_n)\in X\quad\text{such that}\quad
t_1<\cdots<t_n.
\end{empheq}
For every $x$ in $X$, 
there is a unique corresponding piecewise linear function ---
or {\em linear spline} --- $l_{(t,x)}:[t_1,t_n]\to\RR$ given by
\begin{equation}\label{e:l(t,x)}
l_{(t,x)}(s):=x_i+(x_{i+1}-x_i)\frac{s-t_i}{t_{i+1}-t_i},
\quad\text{for}\quad s\in[t_i,t_{i+1}],\ i\in\{1,\ldots,n-1\}.
\end{equation}
In civil engineering, such a spline may represent the \emph{vertical
profile} of a road design. In this context, $t_i$ is the horizontal
distance between a \emph{station} $i\in\{1,\ldots,n-1\}$ along the
road, and the starting station $i=1$ of the same road. The station
value $t_i$, together with the \emph{elevation} value $x_i$ form a
\emph{point of vertical intersection} $(t_i,x_i)$, where two vertical
tangents intersect. Vertical curves are placed beneath or above
these points to allow for a smooth ride.

The most basic problem in road design is to satisfy the following
three types of constraints:
\begin{itemize}
\sepp 
\item {\bf interpolation constraints:} For a subset $J$ of $\{1,\ldots,n\}$, we have $x_j=y_j$, where $y\in\RR^J$ is given.

\item {\bf slope constraints:} each slope $s_j:=(x_{j+1}-x_j)/(t_{j+1}-t_j)$ satisfies $|s_j|\leq\sigma_j$ where $j\in\{1,...,n-1\}$ and $\sigma\in\RPP^{n-1}$ is given.

\item {\bf curvature constraints:} $\gamma_j\geq s_{j+1}-s_j\geq\delta_j$, for every $j\in\{1,\ldots,n-2\}$, and for given $\gamma$ and $\delta$ in $\RR^{n-2}$.
\end{itemize}
The interpolation constraint fixes a point of vertical intersection
$(t_i,x_i)$ to a given elevation $x_i$. This allows for the
construction of an intersection with an existing road that crosses
the new road at $t_i$.  The slope constraint is required for safety
reasons and to ensure good traffic flow.  The curvature constraints
limits the grade change of the incoming and outgoing tangents. This
limits the curvature of vertical smoothing curves, which is very
important for the visibility of oncoming traffic. It also limits
the vertical acceleration on a vehicle, which contributes to a more
comfortable ride.

The engineer is first and foremost concerned with meeting these
constraints. 
In \cite{BK13}, it is shown how the engineer's problem can be
translated into a
feasibility problem involving six sets in $X$:
\begin{equation}
\text{find $x\in C_1\cap C_2\cap \cdots \cap C_6$.}
\end{equation}
Of the infinitude of possible solutions for this problem,
the engineer may be particularly interested in those that are
optimal in some sense. For instance, in road design, it is desirable to 
find a solution that may be close to a given fixed vector
, a solution that minimizes the amount of earth work
(cut and fill), a solution that balances cut and fill, 
or variants and combinations thereof. If more than one
objective function is of interest, it is common to additively
combine these functions, perhaps by scaling the functions to give
different levels of importance to them. In summary, we are faced
with the problem
\begin{empheq}[box=\mybluebox]{equation}\label{e:prob}
{\rm minimize}\quad F(x)
\quad\text{subject to}\quad
x\in C_1\cap\cdots \cap C_6, 
\end{empheq}
where $F$ itself may be a sum of (scaled) objective functions. 
The function $F$ is typically \emph{nonsmooth} which prevents the
use of standard optimization methods. 
This is the abstraction of the road design optimization problem.

\subsection{Objective and outline of this paper}

The objective of this paper is to present a framework for solving
the problem \eqref{e:prob} based on the Douglas--Rachford
splitting algorithm. This involves the introduction
and computation of new proximity operators to deal with the
objective function. Once all required operators are obtained in
closed form, we test the algorithm numerically. 

The Douglas-Rachford algorithm itself will be reviewed 
in Section~\ref{s:DR}. 
The projection operators and proximity operators are obtained 
in Sections~\ref{s:Proxetal}--\ref{s:rarea}. 
We report on numerical experiments in Section~\ref{s:numerix},
which also contains some concluding remarks. 

\subsection{Notation}
We write $\NN$ for the nonnegative integers $\{0,1,2,\ldots\}$
and $\RR$ for the real numbers.
We also set $\RP = \menge{x\in\RR}{x\geq 0}$,
$\RPP = \menge{x\in\RR}{x>0}$, $\RM=-\RP$, and $\RMM=-\RPP$. 
Notation not explicitly defined follows \cite{BC}. 

\section{Proximity operators, projectors, and norms}

\label{s:Proxetal}

\subsection{Projectors}

Let $C$ be a nonempty closed convex subset of $X$.
It is well known (see, e.g., \cite[Theorem~3.14]{BC}) 
that every point $x$ in $X$ has \emph{exactly one}
nearest point in $C$, denoted by $\prox_C(x)$ and called the projection
of $x$ onto $C$. The induced operator 
\begin{equation}
\prox_C \colon X\to X
\end{equation}
is called the \emph{projection operator} or \emph{projector} of
$C$.

The following two projectors are simple but useful.

\begin{example}
\label{ex:clip}
Let $\alpha$, $\beta$, and $x$ be in $\RR$ such that
$\alpha<\beta$.
Then 
\begin{equation}
\prox_{[\alpha,\beta]}(x) = 
\max\big\{ \alpha,\min\{\beta,x\}\big\}
= \min\big\{ \beta,\max\{\alpha,x\}\big\}=
\begin{cases}
\alpha, &\text{if $x<\alpha$;}\\
x, &\text{if $\alpha\leq x\leq \beta$;}\\
\beta, &\text{if $\beta<x$.}
\end{cases}
\end{equation}
Moreover, $\beta-\prox_{[\alpha,\beta]}(x) =
\prox_{[0,\beta-\alpha]}(\beta-x)$;
in particular,
\begin{equation}
\label{e:01clip}
1-\prox_{[0,1]}(x) = \prox_{[0,1]}(1-x).
\end{equation}
\end{example}

\begin{lemma}[projector of a line segment]
\label{l:projseg}
Let $a$ and $b$ be distinct vectors in $X$,
let $x\in X$, and set $q = \scal{a-x}{a-b}/\|a-b\|^2$. 
Then
\begin{equation}
\label{e:projseg}
\prox_{[a,b]}(x) = (1-\lambda)a + \lambda b,
\;\;\text{where}\;\;
\lambda = P_{[0,1]}(q)
= \begin{cases}
0, &\text{if $q< 0$;}\\
q, &\text{if $q\in[0,1]$;}\\
1, &\text{if $q>1$.}
\end{cases}
\end{equation}
Alternatively, and more symmetrically, 
\begin{equation}
\label{e:symmprojseg}
\prox_{[a,b]}(x) = 
\prox_{[0,1]}\bigg(\frac{\scal{b-x}{b-a}}{\|b-a\|^2}\bigg)a +
\prox_{[0,1]}\bigg(\frac{\scal{a-x}{a-b}}{\|a-b\|^2}\bigg)b.
\end{equation}
\end{lemma}
\begin{proof}
This follows by discussing the minimization of the
quadratic function
\begin{equation}
\lambda\mapsto
\|x-((1-\lambda)a+\lambda b)\|^2
= (1-\lambda)\|x-a\|^2 +\lambda\|x-b\|^2
-\lambda(1-\lambda)\|a-b\|^2,
\end{equation}
which has the derivative
$2\lambda\|a-b\|^2 - 2\scal{a-x}{a-b}$.
To obtain \eqref{e:symmprojseg}, use \eqref{e:01clip} and
\eqref{e:projseg}. 
\end{proof}

\subsection{Proximity operators}

Let $f\colon X\to\RX$ be a function that is convex, lower
semicontinuous, and proper\footnote{See, e.g., \cite{Rocky70} and
\cite{BC} for relevant material in Convex Analysis.}. 
Fix $x\in X$. Then it well known (see, e.g.,
\cite[Section~12.4]{BC}) that the function 
\begin{equation}
X\to\RX\colon y\mapsto f(y)+\tfrac{1}{2}\|x-y\|^2
\end{equation}
has a \emph{unique} minimizer which we denote by
$\prox_f(x)$. 
The induced operator
\begin{equation}
\prox_f \colon X\to X
\end{equation}
is called the \emph{proximal mapping} or \emph{proximity
operator} (see \cite{Moreau}) of $f$. 
These operators are important building blocks in 
algorithms for solving optimization problems with nonsmooth
objective functions; see, e.g., 
\cite{BC}, \cite{CPBanff}, 
and the references therein.
Note that if $f$ is the \emph{indicator function} of $C$,
i.e., 
\begin{equation}
\iota_C \colon 
X\to\RX\colon
x\mapsto 
\begin{cases}
0, &\text{if $x\in C$;}\\
+\infty, &\text{otherwise,}
\end{cases}
\end{equation}
then $\prox_f = \prox_C$; thus, proximity operators are
generalizations of projectors. 

We also point out that some algorithms utilize $\prox_{f^*}$, 
the proximity operator of the \emph{Fenchel conjugate} $f^*$ of $f$, which
is defined by $f^*(x^*)  = \sup_{x\in X}(\scal{x^*}{x}-f(x))$ at
$x^*\in X$. 
If $\gamma\in\RPP$, then
(see \cite[Theorem~14.3(ii)]{BC})
\begin{equation}
\label{e:decomp}
(\forall x\in X)\quad
x = \gamma\prox_{\gamma^{-1}f}(\gamma^{-1}x) + \prox_{\gamma
f^*}(x).
\end{equation}

\begin{lemma}
\label{l:0511c}
Let $f\colon X \to\RR$ be convex and positively homogeneous, 
let $\alpha\in\RPP$, let $\gamma\in\RPP$, let $w\in X$, and set
\begin{equation}
h\colon X\to\RR \colon x\mapsto \alpha f(x-w).
\end{equation}
Let $x\in X$. Then 
\begin{equation}
\label{e:l:0511c:1}
\prox_{\gamma h}(x)=
w+\gamma\alpha \prox_{f}\big(\tfrac{x-w}{\gamma\alpha}\big) = 
x-\gamma\alpha\prox_{f^*} (\tfrac{x-w}{\gamma\alpha})
\end{equation}
and
\begin{equation}
\label{e:l:0511c:2}
\prox_{\gamma h^*}(x)= x-\gamma w - \alpha
\prox_f\big(\tfrac{x-\gamma w}{\alpha}\big) = \alpha \prox_{f^*}(\tfrac{x-\gamma w}{\alpha}).
\end{equation}
\end{lemma}
\begin{proof}
Using \eqref{e:decomp}, 
we have
\begin{subequations}
\begin{align}
\prox_{\gamma h}(x)&=\argmin_{y\in X}\big(\tfrac{1}{2}\|y-x\|^2+(\gamma\alpha) f(y-w)\big)\\
&=\argmin_{y\in X}\Big(\tfrac{1}{2}\|\tfrac{y-w}{\gamma\alpha}-\tfrac{x-w}{\gamma\alpha}\|^2+ f(\tfrac{y-w}{\gamma\alpha})\Big)\\
&=w+\gamma\alpha\argmin_{z\in X}\Big(\tfrac{1}{2}\|z-\tfrac{x-w}{\gamma\alpha}\|^2+ f(z)\Big)\\
&=w+\gamma\alpha\prox_f(\tfrac{x-w}{\gamma\alpha})\\
&=w+\gamma\alpha\Big(\tfrac{x-w}{\gamma\alpha}-\prox_{f^*}(\tfrac{x-w}{\gamma\alpha})\Big)\\
&=x-\gamma\alpha\prox_{f^*}(\tfrac{x-w}{\gamma\alpha}),
\end{align}
\end{subequations}
which proves \eqref{e:l:0511c:1}.
To obtain \eqref{e:l:0511c:2},
combine \eqref{e:l:0511c:1} with \eqref{e:decomp}.
\end{proof}

\subsection{Primal and dual norms}

Recall that a \emph{norm} $f$ on $X$ is a convex function
such that $(\forall\alpha\in\RR)$ $f(\alpha x)= |\alpha|f(x)$ and
$f$ vanishes only at the origin. Associated with the norm $f$ are its primal 
and dual closed unit balls which are defined by
\begin{equation}
\uball = \uball(f) = \menge{x\in X}{f(x)\leq 1}
\;\;\text{and}\;\;
\uball_* = \uball_*(f) = \menge{x^*\in X}{\sup\scal{x^*}{\uball}\leq 1},
\end{equation}
respectively.

\begin{lemma}[dual ball]
\label{l:0509c}
Let $f\colon X\to\RR$ be a norm.
Then the dual ball is given by 
\begin{equation}\label{e:l0509c1}
\uball_*
= \conv  \overline{\menge{\nabla f(x)}{f(x)=1 \;\text{and}\;
x\in \dom \nabla f}},
\end{equation}
where $\dom\nabla f$ is the sets of points at which $f$ is differentiable.
\end{lemma}
\begin{proof}
Set $S:=\menge{x\in\RR^n}{f(x)=1}$.
Since $f$ is a norm, we have 
$\partial f(0)=\uball_*$.
Moreover, $0\notin\dom\nabla f$ and
$(\forall x\in\dom\nabla f)$ $\nabla f(\RPP x)=\nabla f(x)$.
It follows that
\begin{subequations}
\begin{align}
\menge{\nabla f(x)}{x\in S \cap \dom \nabla f}
&\subseteq
\menge{\lim \nabla f(x_k)}{0\leftarrow x_k\in \dom\nabla f}\\
&\subseteq \overline{\menge{\nabla f(x)}{x\in S \cap \dom \nabla
f}};
\end{align}
\end{subequations}
consequently,
\begin{equation}
\overline{\menge{\lim \nabla f(x_k)}{0\leftarrow x_k\in \dom\nabla f}}
= \overline{\menge{\nabla f(x)}{x\in S \cap \dom \nabla f}}
\end{equation}
Hence, using
\cite[Theorem~25.6 and Theorem~17.2]{Rocky70}, we deduce that 
\begin{subequations}
\begin{align}
\uball_* &=\partial f(0)\\
&= \overline{\conv \menge{\lim \nabla f(x_k)}{0\leftarrow x_k\in
\dom\nabla f}}\\
&= \conv \overline{\menge{\lim \nabla f(x_k)}{0\leftarrow x_k\in
\dom\nabla f}}\\
&= \conv  \overline{\menge{\nabla f(x)}{x\in S \cap \dom \nabla
f}},
\end{align}
\end{subequations}
as claimed.
\end{proof}

\begin{remark}[dual norm]
Let $f\colon X\to\RR$ be a norm.
It follows from \cite[Section~15]{Rocky70} that
the dual norm $f_*$ can be found by
\begin{equation}
\label{e:dualnorm}
(\forall x^*\in X)\quad
f_*(x^*) = \sup\menge{\scal{x^*}{x}}{f(x)=1}.
\end{equation}
Moreover, if
$S$ is a subset of $X$ such that $\conv S$ is equal to the unit
ball of $f$, then
\begin{equation}
\label{e:dualnorm2}
(\forall x^*\in X)\quad
f_*(x^*) = \sup\menge{\scal{x^*}{x}}{x\in S}.
\end{equation}
\end{remark}

We conclude this section with a proximity operator formula that
will be useful later.

\begin{lemma}
\label{l:fuckingback1}
Let $f\colon X\to\RR$ be a norm, and denote its dual ball by
$\uball_*$.
Let $\alpha$ and $\gamma$ be in $\RPP$, let $w\in X$,
and set $h\colon X\to\RR\colon x\mapsto \alpha f(x-w)$.
Then 
\begin{equation}
(\forall x\in X)\quad 
\prox_{\gamma h}(x)=x-\gamma\alpha \prox_{\uball_*}
(\tfrac{x-w}{\gamma\alpha})
\quad\text{and}\quad
\prox_{\gamma h^*}(x)=\alpha \prox_{\uball_*}(\tfrac{x-\gamma w}{\alpha}).
\end{equation}
\end{lemma}
\begin{proof}
This follows from Lemma~\ref{l:0511c} because $f^*=\iota_{B_*}$
(see, e.g., \cite[Proposition~14.12]{BC}) and
$\prox_{\iota_{B_*}} = \prox_{B_*}$.
\end{proof}

\subsection{A menagerie of proximity operators}

In this section we collect various proximity operators
that relevant for road design optimization. 
We provide a user friendly table, taking into account a scaling
parameter and the Fenchel conjugate.

\begin{theorem}\label{t:comm-prxs}
Let $x\in X$, let $w\in X$, let $\alpha\in\RPP$, 
let $\gamma\in\RPP$, and let $\nu\in\{1,\ldots,n\}$.
Then the formulae in the following table hold\footnote{Here
$\|x\|_1=\sum_{\nu=1}^n|x_\nu|$ denotes the $\ell^1$-norm.}:

\begin{tabular}{l|l}
{\rm Function} $f(x)$ & {\rm Proximity operators} $\prox_{\gamma
f}$ {\rm and} $\prox_{\gamma f^*}$\\[+2mm]
\hline
$\iota_C(x)$ & $\prox_{\gamma f}(x) = \prox_C(x)$.\\
& $\prox_{\gamma f^*}(x) = x- \gamma \prox_C(x/\gamma)$.
\\[1.1ex]
\hline
$\alpha\|x-w\|^2$
&$\prox_{\gamma f}(x)=(1+2\alpha\gamma)^{-1}(x+2\alpha\gamma w)$.\\
&$\prox_{\gamma f^*}(x)=x-\gamma(\gamma+2\alpha)^{-1}(x+2\alpha w)$.
\\[+2mm]
\hline\\[-4mm]
$\alpha\|x-w\|$
&$\prox_{\gamma f}(x)=
	\begin{cases}
		x+\alpha\gamma\displaystyle \frac{w-x}{\|w-x\|},&\text{if } \|w-x\|>\alpha\gamma;\\
		w,&\text{otherwise}.
	\end{cases}$\\[+10mm]
&$\prox_{\gamma f^*}(x)=
	\begin{cases}
		\alpha\displaystyle \frac{x-\gamma w}{\|x-\gamma w\|},&\text{if } \|x-\gamma w\|>\alpha;\\
		x-\gamma w,&\text{otherwise}.
	\end{cases}$
\\[+10mm]
\hline\\[-4mm]
$\alpha\|x-w\|_1$
&$
	\big(\prox_{\gamma f}(x)\big)_\nu
	= \begin{cases}
	x_\nu+\alpha\gamma\displaystyle \frac{w_\nu-x_\nu}{|w_\nu-x_\nu|}, &\text{if
	$|w_\nu-x_\nu|>\alpha\gamma$;}\\
	w_\nu, &\text{otherwise.}
	\end{cases}$\\[+10mm]
&$
	\big(\prox_{\gamma f^*}(x)\big)_\nu
	= \begin{cases}
	\alpha \displaystyle \frac{x_\nu-\gamma w_\nu}{|x_\nu-\gamma w_\nu|}, &\text{if
	$|x_\nu-\gamma w_\nu|>\alpha$;}\\
	x_\nu-\gamma w_\nu, &\text{otherwise.}
	\end{cases}$
\\[+10mm]
\hline\\[-5mm]
{$\alpha\left|\scal{x^*}{x-w}\right|$}
& {$\prox_{\gamma f}(x)
=x-(\gamma\alpha)\prox_{[-1,1]}\Big(\tfrac{\scal{x^*}{x-w}}{\gamma\alpha\|x^*\|^2}\Big)x^*$.}\\[+4mm]
& {$\prox_{\gamma f^*}(x)
=\alpha\prox_{[-1,1]}\Big(\tfrac{\scal{x^*}{x-\gamma w}}{\alpha\|x^*\|^2}\Big)x^*$.}
\end{tabular}

\end{theorem}
\proof
\emph{Case 1:} $f(x) = \iota_C$.\\
The formula for $\prox_{\gamma f}$ is obvious, and the one for
$\prox_{\gamma f^*}$ follows from \eqref{e:decomp}.

\emph{Case 2:} $f(x) = \alpha\|x-w\|^2$.\\
Observe that 
$\gamma f(x) = (2\alpha\gamma)\|x-w\|^2/2$.
Hence 
\cite[Table~10.1.xi]{CPBanff} yields
$\prox_{\gamma f}(x)
= (1+2\alpha\gamma)^{-1}(x+2\alpha\gamma w)$.
and 
$\prox_{\gamma^{-1} f}(\gamma^{-1}x)
= (1+2\alpha\gamma^{-1})^{-1}(\gamma^{-1}x+2\alpha\gamma^{-1} w)
= (\gamma +2\alpha)^{-1}(x+2\alpha w)$.
It now follows from \eqref{e:decomp} that
$\prox_{\gamma f^*}(x) =
x-\gamma\prox_{\gamma^{-1}f}(\gamma^{-1}x)=
x- \gamma (\gamma +2\alpha)^{-1}(x+2\alpha w)$.

\emph{Case 3:} $f(x) = \alpha\|x-w\|$.\\
Since the dual ball of the Euclidean ball is the same as the
(primal) ball, denoted by $B$, 
we conclude from Lemma~\ref{l:fuckingback1} that
\begin{equation}
\prox_{\gamma f}(x) = x-\gamma\alpha
\prox_B\big(\tfrac{x-w}{\gamma\alpha}\big)
\;\;\text{and}\;\;
\prox_{\gamma f^*}(x) = \alpha
\prox_B\big(\tfrac{x-\gamma w}{\alpha}\big). 
\end{equation}
The formulae now follow because $P_B(y)=y/\|y\|$ for every $y\in
X\smallsetminus B$.  

\emph{Case 4:} $f(x) = \alpha\|x-w\|_1$.\\
This follows from Case~3 (applied with $X=\RR$)
and \cite[Proposition~23.16]{BC}.

\emph{Case 5:} $f(x)=\alpha|\scal{x^*}{x-w}|$.\\
Set $f_0:=\left|\scal{x^*}{\cdot}\right|$. Then $f_0$ is convex and positively homogeneous, and
\begin{equation}
f(x)=\alpha f_0(x-w).
\end{equation}
Set $D:= [-x^*,x^*]=\menge{t x^*}{t\in[-1,1]}$. 
Then $f_0^*=\iota_D$,
\begin{equation}
\prox_{f_0^*}(x)=\prox_D(x)=\prox_{[-1,1]}\Big(\tfrac{\scal{x}{x^*}}{\|x^*\|^2}\Big)x^*, 
\end{equation}
and the result follows from Lemma~\ref{l:0511c}.
\endproof

\section{The area between two line segments in $\RR^2$}

Let $\tau>0$, and let $(x_1,x_2)\in\RR^2$. 
Consider the two line segments $[(0,0),(\tau,0)]$ and
$[(0,x_1),(\tau,x_2)]$ in the Euclidean plane. 
We will derive a formula for the area $A(x_1,x_2)$ between these two line segments 
(see Figure~\ref{fig:area}).

\subsection{Area and stadium norm}
\label{ss:geo-area}
\begin{figure}[H]
\centering
\includegraphics[]{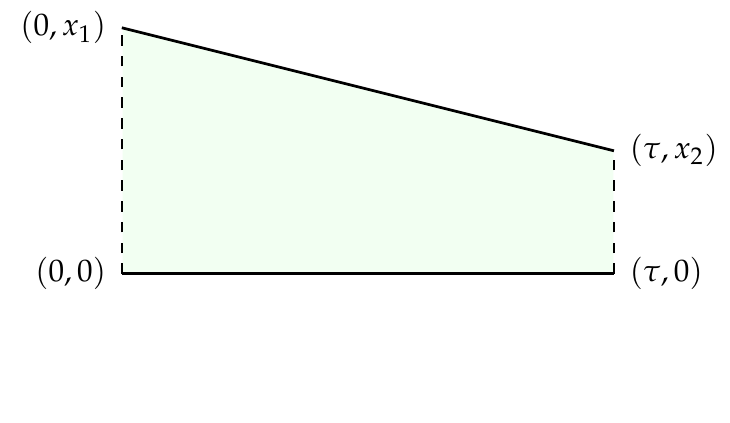}
\quad
\includegraphics[]{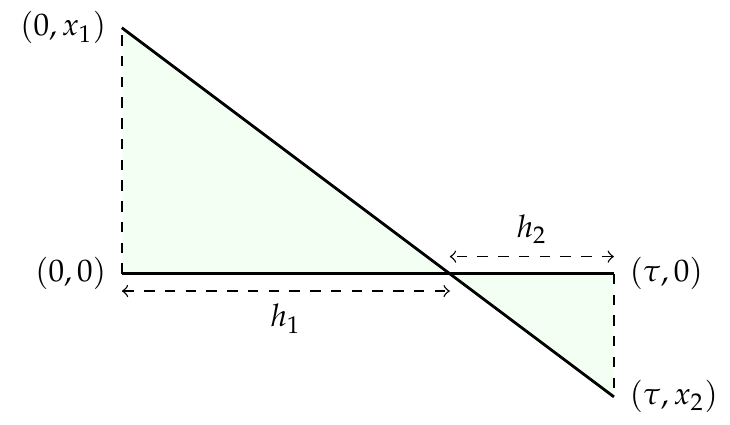}
\caption{Area between two line segments: the two alternatives}
\label{fig:area}
\end{figure}

We consider two cases.

\emph{Case 1}: $x_1x_2\geq 0$. Then it is obvious that
\begin{equation}
A(x_1,x_2)=\frac{\tau}{2}\big(|x_1|+|x_2|\big).
\end{equation}

\emph{Case 2}: $x_1x_2<0$. Then the area consists of two triangles (see Figure~\ref{fig:area}) with heights
\begin{equation}\label{e:0527a}
h_1=\frac{|x_1|\tau}{|x_1|+|x_2|}\quad
\text{and}\quad
h_2=\frac{|x_2|\tau}{|x_1|+|x_2|}.
\end{equation}
Therefore,
\begin{equation}
A(x_1,x_2)=\frac{h_1|x_1|}{2}+\frac{h_2|x_2|}{2}
=\frac{\tau}{2}\Big(\frac{x_1^2+x_2^2}{|x_1|+|x_2|}\Big).
\end{equation}

Combining these two possibilities, we find that 
\begin{equation}\label{e:0728a}
A(x_1,x_2)
=\begin{cases}\displaystyle
\frac{\tau}{2}\bigg(\frac{x_1^2+x_2^2+2\max\{0,x_1x_2\}}{|x_1|+|x_2|}\bigg),&
\text{if}\quad (x_1,x_2)\neq (0,0);\\
0,&\text{otherwise}.
\end{cases}
\end{equation}
Because $\tau$ is fixed, our interest will be in the following function:
\begin{definition}[stadium norm]
The \emph{stadium norm} is defined by
\begin{equation}\label{e:std-0526}
f:\RR^2\to\RR:(x_1,x_2)\mapsto
\begin{cases} \displaystyle
\frac{x_1^2+x_2^2+2\max\{0,x_1x_2\}}{|x_1|+|x_2|},&
\text{if}\quad (x_1,x_2)\neq (0,0);\\
0,&\text{otherwise}.
\end{cases}
\end{equation}
\end{definition}
In fact, one can check that for every $\alpha>0$, the level set
$\menge{x\in\RR^2}{f(x)=\alpha}$ has the geometric shape of a stadium
(see Figure~\ref{fig:lvstd}). This motivates the name \emph{``stadium
norm''};
for the formal proof that $f$ is indeed a norm, see
Section~\ref{s:std-nrm&apprx} below. 

\begin{figure}[H]
\centering
\includegraphics{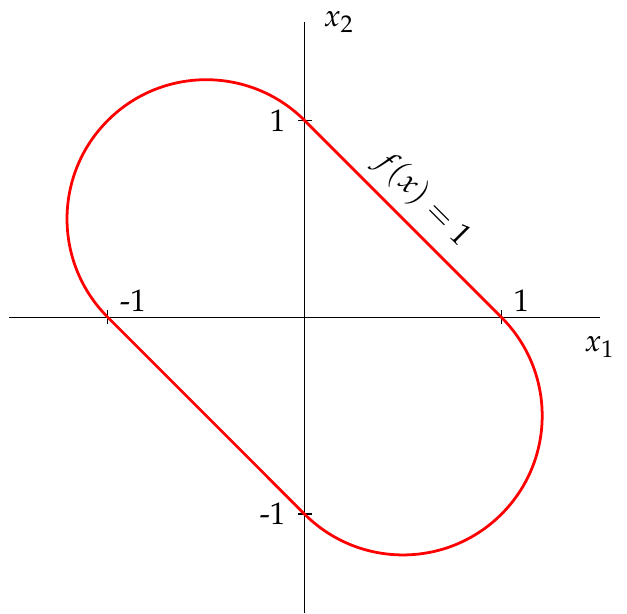}
\caption{A level set of the stadium norm.}
\label{fig:lvstd}
\end{figure}

\subsection{Upper approximations of the area}
\label{ss:challenging}
Since working with the true area \eqref{e:0728a} can be challenging
(see Section~\ref{ss:muchharder} below), we are also interested in simpler
approximations. 
Using the setting of Figure~\ref{fig:area},
we consider two approximations: the classical $\ell^1$-approximation
\begin{equation}\label{e:0526-l1}
A^{\ell}(x)=\frac{\tau}{2}\ell(x)
\quad\text{where}\quad
\ell(x):=\|x\|_1 = |x_1|+|x_2|;
\end{equation}
and the \emph{hexagonal stadium}\footnote{The level set of the function $h$ is a hexagon
(see Figure~\ref{fig:3in1}).} approximation
\begin{equation}\label{e:0526-h}
A^{\ell}(x)=\frac{\tau}{2}h(x_1,x_2)
\quad\text{where}\quad
h(x_1,x_2):=\max\big\{|x_1|,|x_2|,|x_1+x_2|\big\}.
\end{equation}
Both $A^{\ell}$ and $A^{\rm h}$ are \emph{upper approximations},
overestimating the true area $A$: $A^{\ell}\geq A^{\rm h}\geq A$ (see
Figure~\ref{fig:up}).
\begin{figure}[H] 
    \centering 
    \begin{subfigure}{.4\textwidth} 
        \centering 
        \includegraphics[]{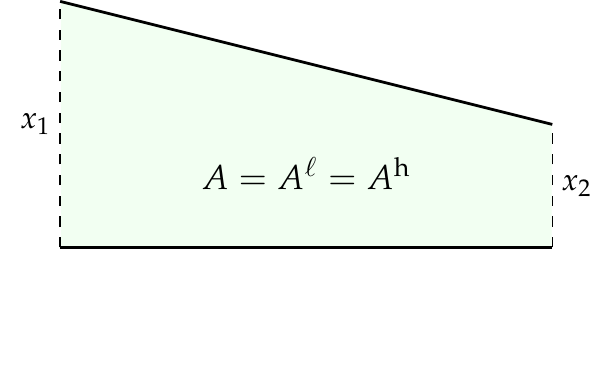} 
        \caption{$x_1x_2\geq 0$} 
        \label{fig:upA} 
    \end{subfigure}%
    \begin{subfigure}{.6\textwidth} 
        \centering 
        \includegraphics[]{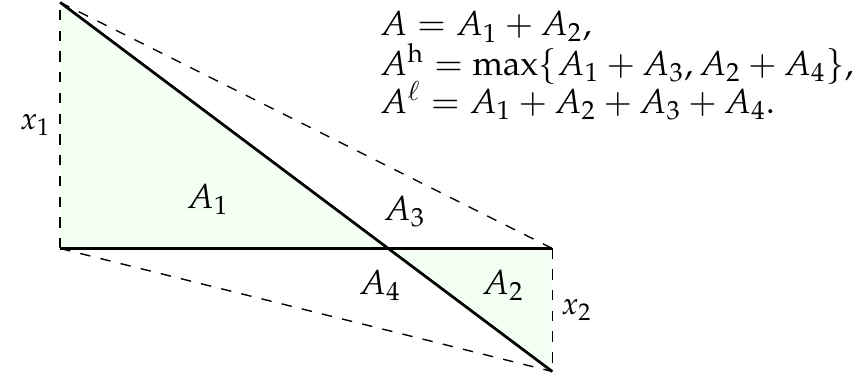} 
        \caption{$x_1x_2<0$} 
        \label{fig:upB} 
    \end{subfigure} 
\caption{$A^{\ell}$ and $A^{\rm h}$ are upper approximations for the area $A$.}
\label{fig:up}
\end{figure} 
In fact, the relationships among $A^{\ell}$, $A^{\rm h}$, and $A$
reflect those among $f$, $\ell$, and $h$, which we turn to now:

\begin{lemma}[upper approximations of the stadium norm]
\label{l:0526a} Consider the stadium norm $f$ from \eqref{e:std-0526}, 
the norm $\ell=\|\cdot\|_1$ from \eqref{e:0526-l1}, 
and the hexagonal stadium norm $h$ from \eqref{e:0526-h}.
Let $x = (x_1,x_2)\in\RR^2$. 
Then 
\begin{subequations}
\begin{equation}\label{e:0413ia}
f(x)\leq h(x)\leq \ell(x)
\end{equation}
and
\begin{equation}
\label{e:0413ia+}
f(x)= h(x)= \ell(x)
\quad\Leftrightarrow\quad
x_1x_2\geq 0.
\end{equation}
\end{subequations}
Moreover, 
\begin{equation}\label{e:0413ib}
\ell(x)-f(x)
\geq 2\big(h(x)-f(x)\big)
\end{equation}
and the constant $2$ is optimal. 
\end{lemma}
\begin{proof}
\eqref{e:0413ia}: 
The second inequality is clear. 
To prove the first one, we consider two cases. 
Case 1: $x_1x_2\geq 0$. Then
$f(x)=|x_1|+|x_2|=|x_1+x_2|=\max\{|x_1|,|x_2|,|x_1+x_2|\}=h(x) =
\ell(x)$. 
Case 2: $x_1x_2<0$. Then 
$f(x)=\tfrac{x_1^2+x_2^2}{|x_1|+|x_2|}
\leq\max\{|x_1|,|x_2|\}\leq h(x)$.

\eqref{e:0413ia+}: 
This follows easily from the definitions.

\eqref{e:0413ib}: 
In view of \eqref{e:0413ia+}, the inequality is trivial
when $x_1x_2\geq 0$.
Thus, we assume that $x_1x_2<0$.
Set $M:=\max\{|x_1|,|x_2|\}$ and $m:=\min\{|x_1|,|x_2|\}$. 
Then
$f(x) = (m^2+M^2)/(m+M)$,
$h(x) = M$ and 
$\ell(x) = m+M$.
Hence if $\beta\in\RPP$, then 
\begin{equation}
\ell(x)-f(x) = m+M - \frac{m^2+M^2}{m+M} =
\frac{2mM}{m+M}
\end{equation}
and 
\begin{equation}
\beta\big(h(x)-f(x)\big)
= \beta\bigg(M - \frac{m^2+M^2}{m+M}\bigg) =
\frac{\beta m(M-m)}{m+M}. 
\end{equation}
This implies \eqref{e:0413ib} and we also conclude that
the constant $2$ is optimal. 
\end{proof}

\subsection{The signed area between two line segments}

\label{ss:signed}

We will now derive a formula for the
signed area $S_\tau(x_1,x_2)$ between two line segments $[(0,0),(\tau,0)]$
and $[(0,x_1),(\tau,x_2)]$ (see Figure~\ref{fig:area2b}). 
Consider, e.g., the case when 
$x_1>0$ and $x_2<0$. Using \eqref{e:0527a}, we have
\begin{equation}
S_\tau(x_1,x_2)=\frac{h_1 |x_1|}{2}-\frac{h_2|x_2|}{2}
=\frac{\tau}{2}\frac{x_1^2-x_2^2}{|x_1|+|x_2|}
=\frac{\tau}{2}\big(|x_1|-|x_2|\big)
=\frac{\tau}{2}(x_1+x_2).
\end{equation}
The remaining cases can be dealt with analogously; 
altogether, we then obtain the following simple formula 
for the signed area between the two line segments:
\begin{equation}
S_\tau(x_1,x_2)=\frac{\tau}{2}(x_1+x_2).
\end{equation}
\begin{figure}[H]
\centering
\includegraphics[]{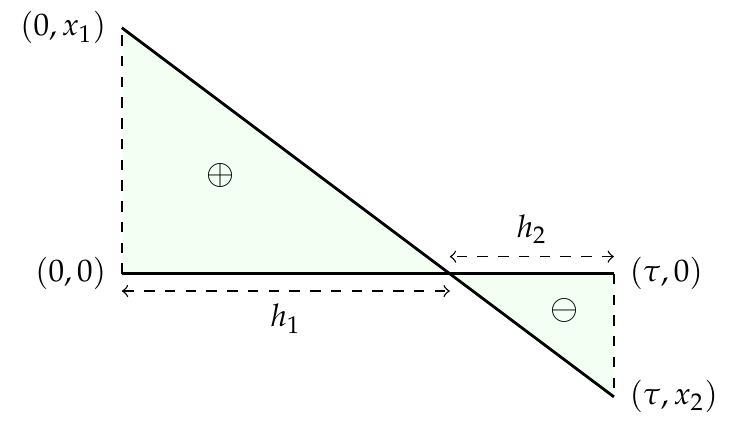}
\caption{Signed area between the two line segments}
\label{fig:area2b}
\end{figure}

\section{The stadium norm and its approximations}

We now justify our naming  convention by showing
that the stadium norm is actually a norm.
(For further recent results on checking convexity of piecewise-defined
functions, see \cite{BLP}.)

\label{s:std-nrm&apprx}

\begin{theorem}[stadium norm is indeed a norm]
\label{t:0526c}
Set 
\begin{equation}
\label{e:std-nrm}
f\colon\RR^2\to\RR\colon
(x_1,x_2)\mapsto
\begin{cases}
\displaystyle
\frac{x_1^2+x_2^2+2\max\{0,x_1x_2\}}{|x_1|+|x_2|},&\text{if
$(x_1,x_2)\neq (0,0)$;}\\
0,&\text{otherwise,}
\end{cases}
\end{equation}
and let $\Omega_1 = \RP\times\RP$, 
$\Omega_2 = \RM\times\RP$, 
$\Omega_3 = \RM\times\RM$, and
$\Omega_4 = \RP\times\RM$ denote the four closed quadrants in the Euclidean
plane. 
Then $f$ is a norm, called the \emph{stadium norm}, 
and continuously differentiable at every point
$(x_1,x_2)\in\RR^2\smallsetminus\{(0,0)\}$
with 
\begin{equation}\label{e:0324a}
\nabla f(x_1,x_2)=\begin{cases}
(1,1),&\text{if $(x_1,x_2)\in\Omega_1$;}\\[+2mm]
\displaystyle
\bigg(\frac{-x_1^2+2x_1x_2+x_2^2}{(x_1-x_2)^2},\frac{-x_1^2-2x_1x_2+x_2^2}{(x_1-x_2)^2}\bigg),
&\text{if $(x_1,x_2)\in \Omega_2$;}\\[+2mm]
(-1,-1), &\text{if $(x_1,x_2)\in\Omega_3$;}\\[+2mm]
\displaystyle
\bigg(\frac{x_1^2-2x_1x_2-x_2^2}{(x_1-x_2)^2},\frac{x_1^2+2x_1x_2-x_2^2}{(x_1-x_2)^2}\bigg),&\text{if
$(x_1,x_2)\in\Omega_4$.}
\end{cases}
\end{equation}
\end{theorem}
\begin{proof}
It is clear that $f$ is continuous and that $f$ is positively
homogeneous. 
The identity \eqref{e:0324a} follows easily from the definition
of $f$. 
Let $x=(x_1,x_2)\in\RR^2\smallsetminus\{(0,0)\}$.
If $x\in\Omega_1\cup\Omega_3$, then $f(x) = |x_1|+|x_2|$;
thus, $f|_{\Omega_1}$ and $f|_{\Omega_3}$ are obviously convex. 
If $x\in\inte\Omega_2$, then the Hessian of $f$ at $x$, 
\begin{equation}
\nabla^2 f(x)=\frac{4}{(x_2-x_1)^3}\begin{pmatrix}
x_2^2&-x_1x_2\\
-x_1x_2&x_1^2
\end{pmatrix}, 
\end{equation}
is positive semidefinite. 
It follows that $f|_{\inte\Omega_2}$ is convex and so
is $f|_{\Omega_2}$ by using the continuity of $f$ 
(see, e.g., \cite[Proposition~17.10 and Proposition~9.26]{BC}). 
The proof of the convexity of $f|_{\Omega_4}$ is similar. 

Now let $y\in\RR^2$ and assume that $(0,0)\notin [x,y]$. 
Then there exist points (not necessarily distinct) points 
$u$ and $v$ in $\RR^2$ such that
\begin{equation}
[x,y]=[x,u]\cup[u,v]\cup[v,y],
\end{equation}
with $[x,u]\subseteq A_1$, $[u,v]\subseteq A_2$, and $[v,y]\subseteq A_3$, 
where $\{A_1,A_2,A_3\}\subseteq \{\Omega_1,\Omega_2,\Omega_3,\Omega_4\}$. 
Note that $f$ is differentiable on $[x,y]$. 
We claim that
\begin{equation}\label{e:p0401a2}
\scal{\nabla f(x)}{y-u}\leq \scal{\nabla f(u)}{y-u}.
\end{equation}
Indeed, \eqref{e:p0401a2} is obvious when $x=u$. 
If $u\neq x$, then, since $f$ is convex in $A_1$, we have
\begin{subequations}
\begin{align}
\scal{\nabla f(x)}{y-u}
&=\tfrac{\|y-u\|}{\|u-x\|}\scal{\nabla f(x)}{u-x}\\
&\leq \tfrac{\|y-u\|}{\|u-x\|}\scal{\nabla f(u)}{u-x}
=\scal{\nabla f(u)}{y-u}.
\end{align}
\end{subequations}
Analogously, we see that 
\begin{equation}
\label{e:p0401a3}
\scal{\nabla f(u)}{y-v}\leq \scal{\nabla f(v)}{y-v}.
\end{equation}
Employing \eqref{e:p0401a2}, \eqref{e:p0401a3}, and the convexity
of $f|_{A_i}$, we deduce 
\begin{subequations}
\begin{align}
\scal{\nabla f(x)}{y-x}
&=\scal{\nabla f(x)}{u-x}+\scal{\nabla f(x)}{y-u}\\
&\leq f(u)-f(x)+\scal{\nabla f(u)}{y-u}\\
&=f(u)-f(x)+\scal{\nabla f(u)}{v-u}+\scal{\nabla f(u)}{y-v}\\
&\leq \big(f(u)-f(x)\big)+\big(f(v)-f(u)\big)+\scal{\nabla f(v)}{y-v}\\
&\leq \big(f(v)-f(x)\big)+\big(f(y)-f(v)\big)\\
&=f(y)-f(x).
\end{align}
\end{subequations}
To summarize, we have proven
\begin{equation}
\label{e:p0401a1}
(0,0)\notin [x,y] \;\;\Rightarrow\;\;
\scal{\nabla f(x)}{y-x} \leq f(y)-f(x).
\end{equation}

Now let $x$ and $y$ be in $\RR^2$ such that $x\neq y$, let
$\lambda\in[0,1]$, and set 
$z = (1-\lambda)x + \lambda y$.
It remains to show that 
\begin{equation}
\label{e:t0401b1}
f(z)\leq(1-\lambda)f(x)+\lambda f(y).
\end{equation}

\emph{Case~1}: $(0,0)\notin[x,y]$.\\
Then $(0,0)\notin[x,z]$ and $(0,0)\notin[z,y]$. 
Applying \eqref{e:p0401a1} twice, we obtain
\begin{equation}
\scal{\nabla f(z)}{x-z}\leq f(x)-f(z)
\quad\text{and}\quad
\scal{\nabla f(z)}{y-z}\leq f(y)-f(z).
\end{equation}
It follows that 
$(1-\lambda)\scal{\nabla f(z)}{x-z}\leq (1-\lambda)(f(x)-f(z))$
and
$\lambda\scal{\nabla f(z)}{y-z}\leq \lambda(f(y)-f(z))$, which
after adding and re-arranging turns into \eqref{e:t0401b1}. 

\emph{Case~2}: $(0,0)\in[x,y]$.\\
Let $w$ be a unit vector perpendicular to $[x,y]$, 
let $\varepsilon\in\RPP$, and set 
\begin{equation}
x_\ve=x+\ve w,\;\;
y_\ve=y+\ve w,
\;\;\text{and}\;\;
z_\ve=z+\ve w.
\end{equation}
It is clear that $(0,0)\not\in[x_\ve,y_\ve]$. 
So, applying Case~1 to $[x_\ve,y_\ve]$, we deduce that
\begin{equation}
f(z_\ve)\leq (1-\lambda)f(x_\ve)+\lambda f(y_\ve).
\end{equation}
Taking the limit as $\ve\to 0^+$ and using the continuity of $f$, we 
obtain \eqref{e:t0401b1}. 
\end{proof}

\begin{proposition}[dual stadium norm]
\label{p:0526d}
Consider the norm
\begin{equation}
g\colon \RR^2\to\RR \colon
(x_1,x_2)\mapsto \tfrac{1}{2}|x_1-x_2| +
\tfrac{1}{\sqrt{2}}\|(x_1,x_2)\|.
\end{equation}
Then the stadium norm $f$ given by \eqref{e:std-nrm} is the norm dual
to $g$.
\end{proposition}
\begin{proof}
Let us sketch the derivation\footnote{We note in passing
that $g$ was not found until after we computed the projection
onto the dual ball of $f$ (see Subsection~\ref{ss:muchharder} below) 
and ``guessed'' the formula for $g$.}. 
It is easy to check that $g$ is indeed a norm. 
Denote the norm dual to $g$ by $g_*$. By 
If $g(\xi,\eta)=1$, then solving for $\eta$ yields
two solutions, namely
\begin{equation}
\eta_\pm(\xi) = -\xi \pm 2\big(\sqrt{2\pm 2\xi}-1\big),
\;\;\text{where $\xi\in[-1,1]$}. 
\end{equation}
Now let $(x_1,x_2)\in\RR^2$. 
Hence, using \eqref{e:dualnorm}, we have
\begin{subequations}
\label{e:researchgate}
\begin{align}
g_*(x_1,x_2) &= \sup\menge{x_1\xi+x_2\eta}{g(\xi,\eta)=1}\\
&= \max\Big\{ \max_{\xi\in[-1,1]}\big(x_1\xi+x_2\eta_+(\xi)\big), 
\max_{\xi\in[-1,1]}\big(x_1\xi+x_2\eta_-(\xi)\big)\Big\}. 
\end{align}
\end{subequations}
This reduces the problem to one-dimensional calculus. 
If $x_1\neq x_2$, then the 
the critical points of the functions
$\xi\mapsto x_1\xi+x_2\eta_+(\xi)$ and 
$\xi\mapsto x_1\xi+x_2\eta_-(\xi)$ 
are 
$\mp(x_1^2-2x_1x_2+x_2^2)/(x_1-x_2)^2$; otherwise
the critical points are the endpoints $\mp 1$. 
Substituting the critical points into 
\eqref{e:researchgate} yields indeed $g_*=f$.
\end{proof}

Let us summarize our finding in
the following result:

\begin{theorem}[the three norms]\label{t:0617a}
The following table summarizes the dual norms found for the 
three planar norms of interest (see also Figure~\ref{fig:3in1}).

\begin{tabular}{l|l|l}
{\rm Norm} $f$ & {\rm Formula for $f(x)$}
&{\rm Formula for $f_*(x)$} \\[+2mm]
\hline
{\rm $\ell = \|\cdot\|_1$} & $|x_1|+|x_2|$.
& $\max\big\{|x_1|,|x_2|\big\}$ \\[+2mm]
\hline
{\rm hexagonal stadium}
& $\max\big\{|x_1|,|x_2|,|x_1+x_2|\big\}$
& $\max\big\{|x_1|,|x_2|,|x_1-x_2|\big\}$ \\[+2mm] 
\hline
 & & \\[-4mm]
{\rm stadium}
& $\displaystyle \frac{x_1^2+x_2^2+2\max\{0,x_1x_2\}}{|x_1|+|x_2|}$
& $\tfrac{1}{2}|x_1-x_2| + \tfrac{1}{\sqrt{2}}\|(x_1,x_2)\|$
\end{tabular}

\end{theorem}
\begin{proof}
\emph{Case 1:} $f=\ell = \|\cdot\|_1$.\\
Of course, this case is well known,
we include the details because it is short and for completeness.
Note that its unit ball is 
$\conv\{\pm(1,0),\pm(0,1)\}$. 
Again \eqref{e:dualnorm2} yields
\begin{subequations}
\begin{align}
f_*(u_1,u_2) &=
\max\menge{u_1x_1+u_2x_2}{(x_1,x_2)\in\{\pm(1,0),\pm(0,1)\}}\\
&=\max\big\{\pm u_1,\pm u_2\big\}\\
&=\max\big\{|u_1|,|u_2|\big\}\\
&=\|(u_1,u_2)\|_\infty.
\end{align}
\end{subequations}

\emph{Case 2:} Hexagonal stadium norm.\\
Here $f(x) = \max\{|x_1|,|x_2|,|x_1+x_2|\}$.
Considering the unit sphere $f(x)=1$,
we compute that the unit ball is
$\conv\{\pm(-1,1),\pm(1,0),\pm(0,1)\}$. 
Now let $(u_1,u_2)\in\RR^2$. It follows from
\eqref{e:dualnorm2} that
\begin{subequations}
\begin{align}
f_*(u_1,u_2) &=
\max\menge{u_1x_1+u_2x_2}{(x_1,x_2)\in\{\pm(-1,1),\pm(1,0),\pm(0,1)\}}\\
&=\max\big\{ \pm(u_2-u_1),\pm u_1,\pm u_2\big\}\\
&=\max\big\{|u_1|,|u_2|,|u_1-u_2|\big\}.
\end{align}
\end{subequations}

\emph{Case 3:} $f$ is the stadium norm --- 
see Theorem~\ref{t:0526c} and Proposition~\ref{p:0526d}.  
\end{proof}

\begin{figure}[H]
\centering
\includegraphics[height=3in]{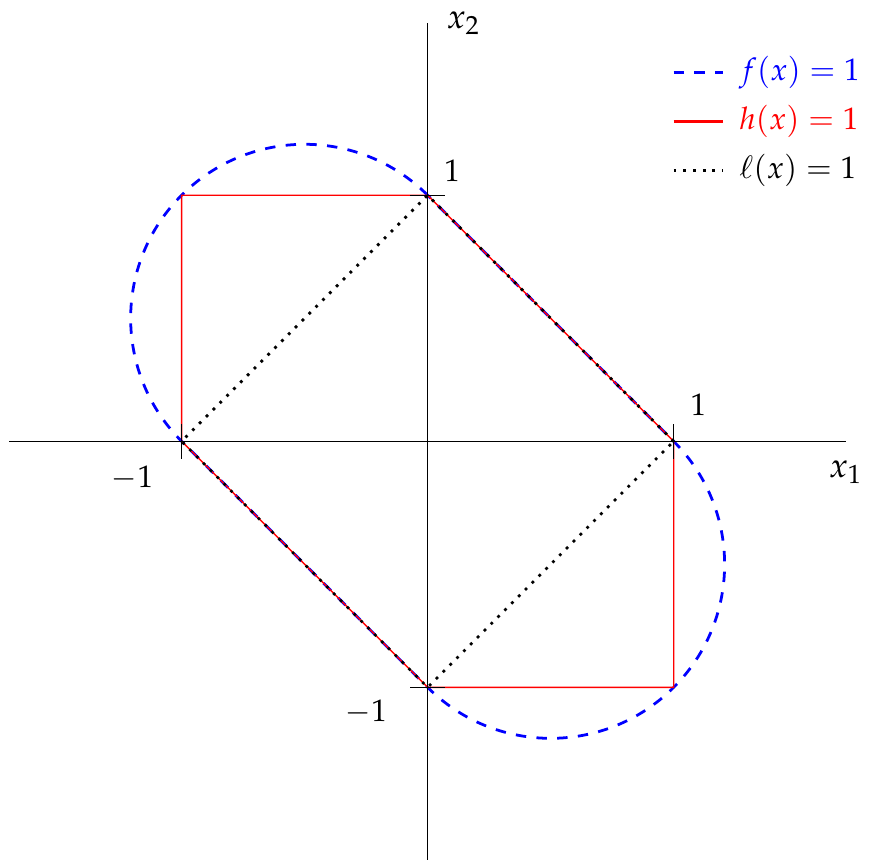}
\includegraphics[height=3in]{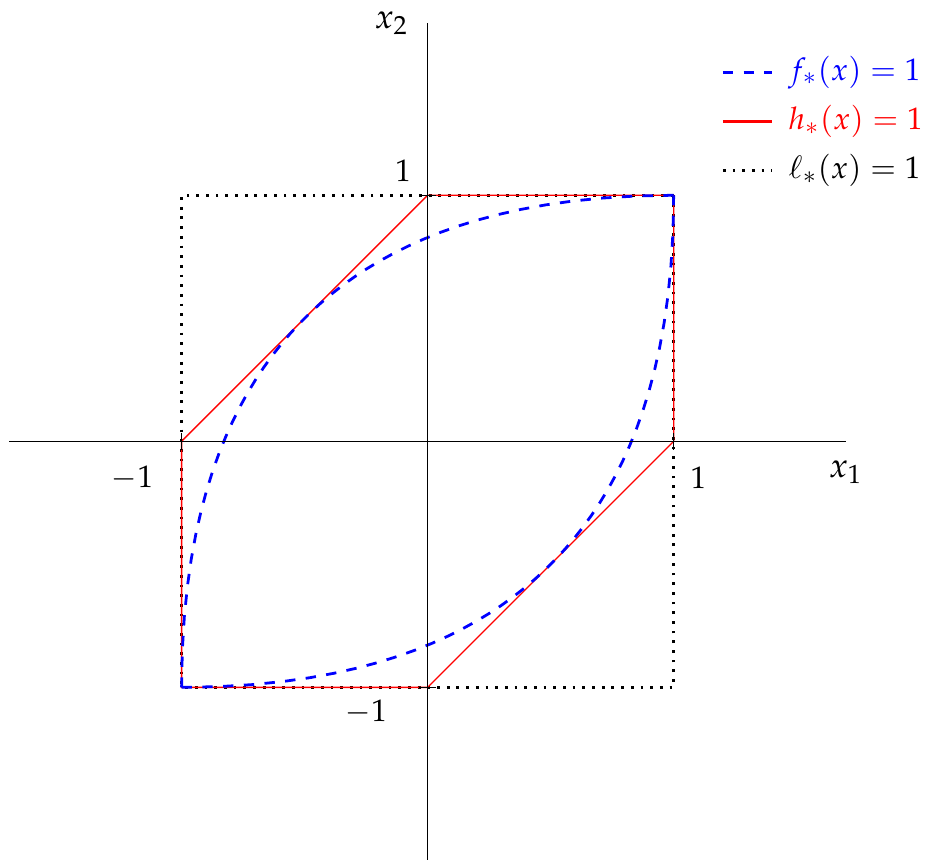}
\caption{Primal and dual balls of the stadium norm $f$, the hexagonal stadium norm $h$, 
and classical $\ell=\|\cdot\|_1$.}
\label{fig:3in1}
\end{figure}

\section{Proximity operators of some planar norms}

\subsection{Projectors onto the dual balls for two polyhedral
norms}

The following result is well known.

\begin{proposition}[dual $\|\cdot\|_1$ ball projector]
Let $\|\cdot\|_1\colon (x_1,x_2)\to|x_1|+|x_2|$
be the $\ell^1$ norm on $\RR^2$, 
denote its dual ball $[-1,1]\times[-1,1]$ by 
$\uball_*$, and 
let $x=(x_1,x_2)\in \RR^2$.
Then
\begin{equation}
\label{e:maxproj}
\prox_{\uball_*}(x_1,x_2)
=\big(\prox_{[-1,1]}(x_1),\prox_{[-1,1]}(x_1)\big).
\end{equation}
\end{proposition}

\begin{proposition}[dual hexagonal stadium ball projector] \ \\
Let $ (x_1,x_2)\mapsto \max\{|x_1|,|x_2|,|x_1+x_2|\}$ be the hexagonal
stadium norm, 
denote its dual ball by $\uball_*$, 
and let $x=(x_1,x_2)\in\RR^2$. 
Then
\begin{equation}
\label{e:50s1}
\prox_{\uball_*}(x)=
\begin{cases}
x, &\text{if $x\in \uball_*$;}\\
\big(\prox_{[0,1]}(x_1),\prox_{[0,1]}(x_2)\big),
&\text{if $(x_1,x_2)\in(\RP\times\RP)\smallsetminus \uball_*$;}\\
\prox_{[(-1,0),(0,1)]}(x),
&\text{if $(x_1,x_2)\in(\RM\times\RP)\smallsetminus \uball_*$;}\\
\big(\prox_{[-1,0]}(x_1),\prox_{[-1,0]}(x_2)\big),
&\text{if $(x_1,x_2)\in(\RM\times\RM)\smallsetminus \uball_*$;}\\
\prox_{[(0,-1),(1,0)]}(x),
&\text{if $(x_1,x_2)\in(\RM\times\RP)\smallsetminus \uball_*$.}
\end{cases}
\end{equation}
Alternatively (and better suited to programming), 
we have 
\begin{equation}
\label{e:50s3}
\prox_{\uball_*}(x) =
\begin{cases}
x, &\text{if $f_*(x)\leq 1$;}\\
\big(\prox_{[-1,1]}(x_1),\prox_{[-1,1]}(x_2)\big), &\text{else if
$x_1x_2\geq 0$;}\\
\sgn(x_1)\big(\tfrac{1}{2},-\tfrac{1}{2}\big)
+\prox_{[-1,1]}(x_1+x_2)\big(\tfrac{1}{2},\tfrac{1}{2}\big),
&\text{else.}
\end{cases}
\end{equation}
\end{proposition}
\begin{proof}
Formula~\eqref{e:50s1} follows from observing that
\begin{equation}
\uball_* = \conv\big\{ \pm(1,0),\pm(0,1),\pm(1,1)\big\},
\end{equation}
and by considering each quadrant. 
To obtain \eqref{e:50s3}, consider cases
and use \eqref{e:projseg}. 
\end{proof}

\subsection{Projector onto the dual ball of the stadium norm}

\label{ss:muchharder}

In this section, we derive the projector onto the dual ball of
the stadium norm. This will require significantly more work
than the two polyhedral norms just discussed.
We start by setting
\begin{subequations}
\begin{empheq}[box=\mywhitebox]{align}
\Gamma_0 \colon  \big[-\tfrac{\pi}{2},0\big] &\to \RR^2\\
 t &\mapsto 
\Big( \frac{\cos^2 t-2\sin t\cos t-\sin^2 t}{(\cos t-\sin t)^2},
\frac{\cos^2 t+2\sin t\cos t-\sin^2 t}{(\cos t-\sin t)^2}\Big). 
\end{empheq}
\end{subequations}
and
\begin{equation}
R = \ran \Gamma_0.
\end{equation}
In view of Lemma~\ref{l:0509c} and \eqref{e:0324a}, it follows
that 
\begin{equation}
\uball_*=\conv\big(R \cup (-R)\big).
\end{equation}
Using trigonometric identities, we see that
for every $t\in[-\pi/2,0]$ we have
\begin{equation}
\Gamma_0(t) 
=\Big(
\frac{\cos 2t-\sin 2t}{1-\sin2t},
\frac{\cos 2t+\sin 2t}{1-\sin2t}\Big)
=\Big(
\frac{\sqrt{2}\cos (2t+\frac{\pi}{4})}{1-\sin2t},
\frac{\sqrt{2}\sin (2t+\frac{\pi}{4})}{1-\sin2t}\Big).
\end{equation}
By changing variables, we thus see that 
\begin{subequations}
\begin{empheq}[box=\mywhitebox]{align}
\label{e:0413f}
\Gamma_1 \colon  \big[-\tfrac{3\pi}{4},\tfrac{\pi}{4}\big] &\to \RR^2\\
 t &\mapsto 
 \Big(
\frac{\sqrt{2}\cos t}{1-\sin(t-\pi/4)},
\frac{\sqrt{2}\sin t}{1-\sin(t-\pi/4)}\Big)\\
&=
\frac{\sqrt{2}}{1+\cos(t+\pi/4)}\big(\cos t,\sin t\big)
\end{empheq}
\end{subequations}
satisfies
\begin{equation}
R = \ran \Gamma_1.
\end{equation}
In \emph{polar coordinates} $(r,\omega)$,
the parametrizations of $\Gamma_1$ and $-\Gamma_1$ become
\begin{subequations}
\begin{align}
(\Gamma_1)&:\quad
r=\frac{\sqrt{2}}{1+|\cos(\omega+\pi/4)|}=\frac{\sqrt{2}}{1+\cos(\omega+\pi/4)},
\quad \omega\in\big[-\tfrac{3\pi}{4},\tfrac{\pi}{4}\big];\\
(-\Gamma_1)&:\quad
r=\frac{\sqrt{2}}{1+|\cos(\omega+\pi/4)|}=\frac{\sqrt{2}}{1-\cos(\omega+\pi/4)},
\quad \omega\in\big[\tfrac{\pi}{4},\tfrac{5\pi}{4}\big].
\end{align}
\end{subequations}
Now set 
\begin{subequations}
\begin{align}
A_1&:=\menge{(x_1,x_2)\in\RR^2}{x_1\geq 1,x_2\geq 1},\\
A_2&:=\menge{(x_1,x_2)\in\RR^2}{x_1>-1,x_2<1,x_2<x_1}.
\end{align}
\end{subequations}
Then, for every 
$x=(r\cos\omega,r\sin\omega)\in\RR^2$, 
we have 
\begin{equation}
\label{e:0413g}
P_{\uball_*}(x)=
\begin{cases}
(1,1),&\text{if $ x\in A_1$;}\\[+2mm]
(-1,-1),&\text{if $x\in -A_1$;}\\[+2mm]
x,&\text{if $\displaystyle
r\leq\frac{\sqrt{2}}{1+|\cos(\omega+\pi/4)|}$;}\\[+2mm]
\prox_R(x),&\text{if $x\in A_2$ ~and~
$\displaystyle r>\frac{\sqrt{2}}{1+\cos(\omega+\pi/4)}$;}\\[+2mm]
\prox_{-R}(x)=-\prox_{R}(-x),&\text{if $x\in -A_2$ ~and~ 
$\displaystyle r>\frac{\sqrt{2}}{1-\cos(\omega+\pi/4)}$.}
\end{cases}
\end{equation}
For a sketch, see Figure~\ref{f:2}.
\begin{figure}[H]
\centering
\includegraphics{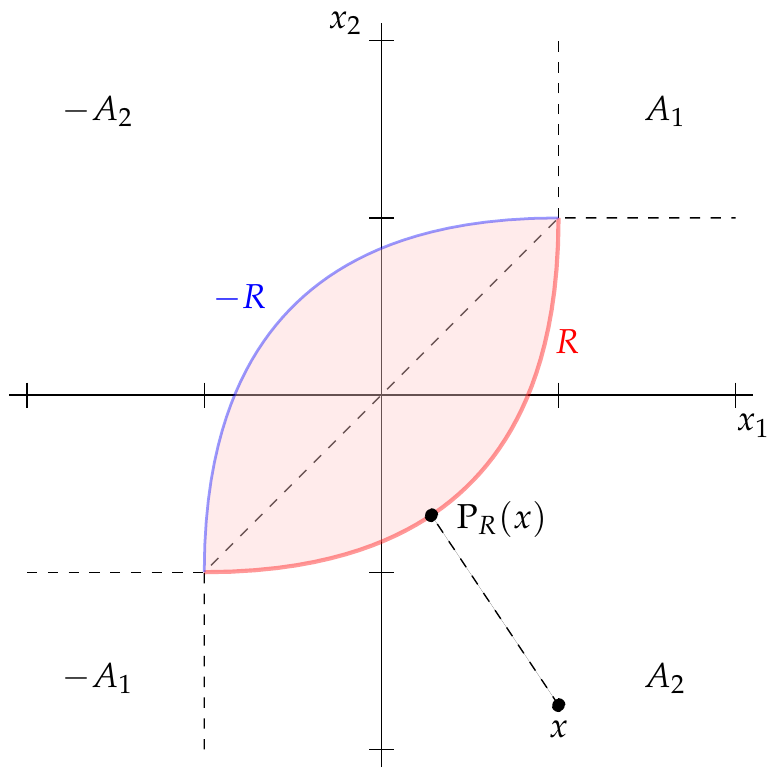}
\caption{Projection onto the dual stadium ball}
\label{f:2}
\end{figure}
Now suppose that 
\begin{equation}
\label{e:0510b1}
x\in A_2 \;\;\text{and}\;\;
r>\frac{\sqrt{2}}{1+\cos(\omega+\pi/4)}.
\end{equation}
Since $x\in A_2$, we have 
$\omega\in\left]-3\pi/4,\pi/4\right[$ and
thus $\cos(\omega+\pi/4)> 0$.
Denote the squared distance 
from $x=(r\cos\omega,r\sin\omega)$ to $\Gamma_1(t)$, where
$t\in[-3\pi/4,\pi/4]$ (see \eqref{e:0413f}) by 
\begin{align}
F(t)&=\Big(\tfrac{\sqrt{2}\cos t}{1+\cos(t+\pi/4)}-r\cos\omega\Big)^2+
\Big(\tfrac{\sqrt{2}\sin t}{1+\cos(t+\pi/4)}-r\sin\omega\Big)^2\\
&=2\Big(\tfrac{\cos t}{1+\cos(t+\pi/4)}
-\tfrac{r\cos\omega}{\sqrt{2}}\Big)^2+
2\Big(\tfrac{\sin t}{1+\cos(t+\pi/4)}-\tfrac{r\sin\omega}{\sqrt{2}}\Big)^2
\end{align}
We now claim that 
\begin{equation}\label{e:0510a}
\begin{cases}
&\text{$F$ is a convex function on $[-3\pi/4,\pi/4]$, and}\\
&\text{$F'(t)=0$ has a unique solution in
$\left]-3\pi/4,\pi/4\right[$.}
\end{cases}
\end{equation}
The critical number $t$ will then yield the  projection $P_R(x) =
\Gamma_1(t)$. 
We start by computing the derivative of $F$:
Indeed, 
\begin{subequations}
\begin{align}
F'(t)&=4\Big(\tfrac{\cos t}{1+\cos(t+\pi/4)}-\tfrac{r\cos\omega}{\sqrt{2}}\Big)
\tfrac{-\sin t(1+\cos(t+\pi/4))
+\sin(t+\pi/4)\cos t}{(1+\cos(t+\pi/4))^2}\\
&\quad+
4\Big(\tfrac{\sin t}{1+\cos(t+\pi/4)}-\tfrac{r\sin\omega}{\sqrt{2}}\Big)
\tfrac{\cos t(1+\cos(t+\pi/4))
+\sin(t+\pi/4)\sin t}{(1+\cos(t+\pi/4))^2}\\
&=4\Big(\tfrac{\cos t}{1+\cos(t+\pi/4)}-\tfrac{r\cos\omega}{\sqrt{2}}\Big)
\tfrac{-\sin t+\sin(\pi/4)}{(1+\cos(t+\pi/4))^2}\\
&\quad+
4\Big(\tfrac{\sin t}{1+\cos(t+\pi/4)}-\tfrac{r\sin\omega}{\sqrt{2}}\Big)
\tfrac{\cos t+\cos(\pi/4)}{(1+\cos(t+\pi/4))^2}\\
&=\tfrac{4}{(1+\cos(t+\pi/4))^2}
\Big(
\tfrac{\sin(t+\pi/4)}{1+\cos(t+\pi/4)}
+\tfrac{r(\sin(t-\omega)-\sin(\omega+\pi/4))}{\sqrt{2}}
\Big)
\end{align}
\end{subequations}
Setting 
\begin{equation}\label{e:0413e}
u:=t+\tfrac{\pi}{4}\in \left[-\tfrac{\pi}{2},\tfrac{\pi}{2}\right]
\quad\text{and}\quad\theta:=\omega+\tfrac{\pi}{4} \in\left]-\tfrac{\pi}{2},\tfrac{\pi}{2}\right[,
\end{equation}
we see that 
\begin{subequations}
\begin{align}
F'(t)&=\tfrac{4}{1+\cos u}\Big(\tfrac{\sin u}{(1+\cos u)^2}+
\tfrac{r}{\sqrt{2}}\tfrac{\sin(u-\theta)-\sin\theta}{(1+\cos u)}\Big)\\
&=\tfrac{4}{1+\cos u}\Big(\tfrac{\sin u}{(1+\cos u)^2}+
\tfrac{r}{\sqrt{2}}\tfrac{\sin u\cos \theta-\cos u\sin\theta-\sin\theta}{(1+\cos u)}\Big)\\
&=\tfrac{4}{1+\cos u}\Big(\tfrac{\sin u}{(1+\cos u)^2}+
\tfrac{r}{\sqrt{2}}\tfrac{\sin u\cos \theta}{(1+\cos u)}-\tfrac{r\sin\theta}{\sqrt{2}}\Big).
\end{align}
\end{subequations}
Furthermore, set 
\begin{equation}\label{e:0413d}
s:=\tan(u/2)\in[-1,1],\ \alpha:=\tfrac{r}{\sqrt{2}}\cos\theta>0,\ \text{and}\ \beta:=\tfrac{r}{\sqrt{2}}\sin\theta.
\end{equation}
Then $\tfrac{1}{1+\cos u}=\tfrac{1}{2\cos^2(u/2)}=\tfrac{1+s^2}{2}$ and
\begin{align}
F'(t)&=\tfrac{4}{2\cos^2(u/2)}
\Big(\tfrac{2\sin (u/2)\cos(u/2)}{4\cos^4(u/2)}
+\tfrac{2\alpha\sin(u/2)\cos(u/2)}{2\cos^2(u/2)}
-\beta\Big)\\
&=2(1+s^2)\Big(\tfrac{1}{2}s(1+s^2) +\alpha s-\beta\Big)\\
&=(1+s^2)\big(s^3+(1+2\alpha)s-2\beta\big). 
\end{align}
Let
\begin{equation}
G\colon[-1,1]\to\RR\colon s\mapsto (1+s^2)\big(s^3+(1+2\alpha)s-2\beta\big)
\end{equation}
and consider the equation 
\begin{equation}\label{e:0413a}
G(s)=0. 
\end{equation}
Since $x=(x_1,x_2)\in A_2$, we have $x_1=r\cos\omega> -1$ and 
$x_2=r\sin\omega< 1$. 
Hence
\begin{subequations}\label{e:0413a2}
\begin{align}
\alpha-\beta &= \tfrac{r}{\sqrt{2}}(\cos\theta-\sin\theta)
=-r\sin\omega>-1,\\
\alpha+\beta &=\tfrac{r}{\sqrt{2}}(\cos\theta+\sin\theta)
=r\cos\omega>-1;
\end{align}
\end{subequations}
consequently, 
\begin{equation}
G(1)=2(2+2\alpha-2\beta)>0
\quad\text{and}\quad
G(-1)=2(-2-2\alpha-2\beta)<0.
\end{equation}
Since $G$ is clearly continuous, it follows that 
\eqref{e:0413a} has a solution in $\left]-1,1\right[$.
We now compute
\begin{equation}
G'(s)=5s^4+6(1+\alpha)s^2-4\beta s+(1+2\alpha) 
\end{equation}
and observe that the discriminant of the quadratic polynomial
$6(1+\alpha)s^2-4\beta s+(1+2\alpha)$ is
$\Delta := 16\beta^2-24(1+\alpha)(1+2\alpha)$.
Because $|\beta|<1+\alpha<1+2\alpha$ (by \eqref{e:0413a2}), 
it is clear that $\Delta<0$. 
Hence $6(1+\alpha)s^2-4\beta s+(1+2\alpha) > 0$ and therefore
$G'$ is strictly positive on $\left]-1,1\right[$. 
We deduce that $G$ is \emph{strictly increasing} on $[-1,1]$.
So the solution of \eqref{e:0413a} is unique. 
In turn, this implies that 
$F'$ strictly increases on
$[-3\pi/4,\pi/4]$. 
It follows that $F$ is a convex function on $[-3\pi/4,\pi/4]$
and that $F'(t)=0$ has a unique solution in $\left]-3\pi/4,\pi/4\right[$.
Therefore, $F$ has a unique minimizer
in $\left]-3\pi/4,\pi/4\right[$, which establishes our claim 
\eqref{e:0510a}.

Now let $s$ be the unique solution of \eqref{e:0413a}, which
implies that $s$ is a real solution of
\begin{equation}
\label{e:0413b}
s^3+(1+2\alpha)s-2\beta=0.
\end{equation}
This real solution is unique because viewed as function in $s$,
the derivate of the left-hand side of \eqref{e:0413b} is 
$3s^2+(1+2\alpha)>0$ since $\alpha>0$.
Cardano's formula gives 
\begin{equation}\label{e:0413c}
\sqrt[3]{\beta+\sqrt{\beta^2+(\tfrac{1+2\alpha}{3})^3}}
+\sqrt[3]{\beta-\sqrt{\beta^2+(\tfrac{1+2\alpha}{3})^3}}
\end{equation}
as a solution to \eqref{e:0413b}.
This solution is a real number, again since $\alpha>0$.
Hence $s$ is equal to \eqref{e:0413c}. 
Let us summarize what we have found out so far:
If 
\begin{subequations}
\label{e:sofar1}
\begin{equation}
\label{e:sofar1a}
x = r(\cos\omega,\sin\omega) \in A_2 \;\;\text{and}\;\;
r>\frac{\sqrt{2}}{1+\cos(\omega+\pi/4)},
\end{equation}
and 
\begin{align}
\alpha&=\tfrac{r}{\sqrt{2}}\cos(\omega+\pi/4),\;\;
\label{e:sofar1b}
\beta=\tfrac{r}{\sqrt{2}}\sin(\omega+\pi/4),\\[+2mm]
\label{e:sofar1c}
s&=\sqrt[3]{\beta+\sqrt{\beta^2+(\tfrac{1+2\alpha}{3})^3}}
+\sqrt[3]{\beta-\sqrt{\beta^2+(\tfrac{1+2\alpha}{3})^3}}\
\in\left]-1,1\right[,\\[+2mm]
\label{e:sofar1d}
t&=2\arctan(s)-\tfrac{\pi}{4}\in\left]-\tfrac{3\pi}{4},\tfrac{\pi}{4}\right[,
\end{align}
then
\begin{equation}
\label{e:sofar1e}
P_R(x)=\frac{\sqrt{2}}{1+\cos(t+\pi/4)}(\cos t,\sin t).
\end{equation}
\end{subequations}
Our next goal is to simplify \eqref{e:sofar1} by eliminating
the trigonometric functions.
To this end,
let $(x_1,x_2)=r(\cos\omega,\sin\omega)\in A_2$.
We translate \eqref{e:sofar1} to a form that is free of
trigonometric functions. 
Observe first that 
\begin{subequations}
\label{e:0510c}
\begin{align}
&r\cos(\omega+\pi/4)=\tfrac{1}{\sqrt{2}}r\cos\omega
-\tfrac{1}{\sqrt{2}}r\sin\omega=\tfrac{1}{\sqrt{2}}(x_1-x_2)>0\\
\text{and}\quad
&r\sin(\omega+\pi/4)=\tfrac{1}{\sqrt{2}}r\cos\omega
+\tfrac{1}{\sqrt{2}}r\sin\omega=\tfrac{1}{\sqrt{2}}(x_1+x_2).
\end{align}
\end{subequations}
Hence \eqref{e:sofar1b} turns into
\begin{equation}
\alpha = \tfrac{1}{2}(x_1-x_2)
\;\;\text{and}\;\;
\beta= \tfrac{1}{2}(x_1+x_2).
\end{equation}
Furthermore, since
\begin{equation}
\frac{\sqrt{2}}{r\big(1+\cos(\omega+\pi/4)\big)}
=\frac{\sqrt{2}}{r+r\cos(\omega+\pi/4)}
=\frac{2}{\sqrt{2(x_1^2+x_2^2)}+x_1-x_2},
\end{equation}
we see that the inequality in \eqref{e:sofar1a} is equivalent to 
\begin{equation}
\label{e:0510b2}
\sqrt{2(x_1^2+x_2^2)}+x_1-x_2>2.
\end{equation}
Next, let $s$ and $t$ be as in
\eqref{e:sofar1c}--\eqref{e:sofar1d}.
Using
\begin{equation}
\cos(\arctan s)=\frac{1}{\sqrt{1+s^2}}\quad\text{and}\quad
\sin(\arctan s)=\frac{s}{\sqrt{1+s^2}},
\end{equation}
we have
\begin{subequations}
\begin{align}
\cos(t+\pi/4)&=\cos(2\arctan s)
=\cos^2(\arctan s)-\sin^2(\arctan s)
=\frac{1-s^2}{1+s^2},\\
\sin(2\arctan s)
&=2\sin(\arctan s)\cos(\arctan s)
=\frac{2s}{1+s^2}.
\end{align}
\end{subequations}
It follows that
\begin{subequations}
\begin{align}
\cos t
&=\tfrac{1}{\sqrt{2}}\big(\cos(2\arctan s)+\sin(2\arctan s)\big)
=\frac{1+2s-s^2}{\sqrt{2}(1+s^2)},\\
\sin t
&=\tfrac{1}{\sqrt{2}}\big(\sin(2\arctan s)-\cos(2\arctan s)\big)
=\frac{-1+2s+s^2}{\sqrt{2}(1+s^2)}.
\end{align}
\end{subequations}
Finally, \eqref{e:sofar1e} turns into
\begin{equation}
\prox_R(x)=\frac{\sqrt{2}}{1+\tfrac{1-s^2}{1+s^2}}
\bigg(\frac{1+2s-s^2}{\sqrt{2}(1+s^2)},
\frac{-1+2s+s^2}{\sqrt{2}(1+s^2)}\bigg)=
\bigg(\frac{1+2s-s^2}{2},\frac{-1+2s+s^2}{2} \bigg).
\end{equation}
Since $\prox_{-R}(x) = -\prox_{R}(-x)$, we can handle the case
when $-x\in A_2$ analogously. 

We are now in a position to summarize this section in
the following result:

\begin{theorem}[dual stadium ball projector]
\label{t:dualBf}
Let 
\begin{equation}
\label{e:dualBf}
f\colon\RR^2\to\RR\colon
(x_1,x_2)\mapsto
\begin{cases}
\displaystyle
\frac{x_1^2+x_2^2+2\max\{0,x_1x_2\}}{|x_1|+|x_2|},&\text{if
$(x_1,x_2)\neq (0,0)$;}\\
0,&\text{otherwise,}
\end{cases}
\end{equation}
be the stadium norm, denote its dual ball by $\uball_*$,
and let $x=(x_1,x_2)\in\RR^2$. 
Set 
\begin{subequations}
\begin{equation}
\alpha:=\tfrac{1}{27}\big(1+|x_1-x_2|\big)^3,\;\;
\beta:=\tfrac{1}{2}\sgn(x_1-x_2)(x_1+x_2),
\end{equation}
and 
\begin{equation}
s:=\sqrt[3]{\beta+\sqrt{\beta^2+\alpha}}
+\sqrt[3]{\beta-\sqrt{\beta^2+\alpha}}.
\end{equation}
\end{subequations}
Then
\begin{equation}
\label{e:stadiumdualproj}
\prox_{\uball_*}(x)=
\begin{cases}
(1,1),&\text{if $x_1\geq 1$ and $x_2\geq 1$;}\\
(-1,-1),&\text{if $x_1\leq -1$ and $x_2\leq -1$;}\\
(x_1,x_2),&\text{if $\sqrt{2(x_1^2+x_2^2)}+|x_1-x_2|\leq 2$;}\\
\displaystyle \sgn(x_1-x_2)\bigg(\frac{1+2s-s^2}{2},\frac{-1+2s+s^2}{2}
\bigg),&\text{otherwise.}
\end{cases}
\end{equation}
\end{theorem}

\subsection{Proximity operators}

Combining 
Lemma~\ref{l:fuckingback1} with the
formulae derived with in
\eqref{e:maxproj}, \eqref{e:50s1}--\eqref{e:50s3},
and \eqref{e:stadiumdualproj},
we are now able to summarize the findings of this section.

\begin{theorem}[planar proximity operators]
\label{t:pla-prx}
Let $f\colon \RR^2\to\RR$ be a norm, and denote its dual ball by
$\uball_*$.
Let $\alpha$ and $\gamma$ be in $\RPP$, let $w\in X$,
and set $h\colon X\to\RR\colon x\mapsto \alpha f(x-w)$.
Then 
\begin{equation}
(\forall x\in X)\quad 
\prox_{\gamma h}(x)=x-\gamma\alpha \prox_{\uball_*}
(\tfrac{x-w}{\gamma\alpha})
\quad\text{and}\quad
\prox_{\gamma h^*}(x)=\alpha \prox_{\uball_*}(\tfrac{x-\gamma w}{\alpha}).
\end{equation}
Let $x\in\RR^2\smallsetminus\{(0,0)\}$. 
The following table summarizes several choices that will be used
later.

\begin{tabular}{l|l|l}
{\rm Norm} $f$ & {\rm Formula for $f(x)$}
&{\rm Formula for $\prox_{\uball_*}(x)$} \\[+2mm]
\hline
{\rm $\ell = \|\cdot\|_1$} & $|x_1|+|x_2|$.
& {\rm see \eqref{e:maxproj}} \\[+2mm]
\hline
{\rm hexagonal stadium}
& $\max\big\{|x_1|,|x_2|,|x_1+x_2|\big\}$
& {\rm see \eqref{e:50s1} or \eqref{e:50s3}} \\[+2mm] 
\hline
 & & \\[-4mm]
{\rm stadium}
& $\displaystyle \frac{x_1^2+x_2^2+2\max\{0,x_1x_2\}}{|x_1|+|x_2|}$
& {\rm see \eqref{e:stadiumdualproj}}
\end{tabular}
\end{theorem}

\section{Proximity operators in $\RR^n$ related to area}

\label{s:rarea}

Let 
\begin{equation}
t=(t_1,\ldots,t_n)\in X=\RR^n\quad\text{such that}\quad
t_1<\cdots<t_n.
\end{equation}
Fix $w=(w_1,\ldots,w_n)\in X$ and let $x=(x_1,\ldots,x_n)\in X$.
In this section, we first develop a formula for the area between the two
linear splines $l_{(t,x)}$ and $l_{(t,w)}$ (see
\eqref{e:l(t,x)}) and then provide related proximity operators.
We set 
\begin{subequations}\label{e:tau-eta}
\begin{align}
&\tau_i:=(t_{i+1}-t_i)/2\quad\text{for}\quad i\in\{1,\ldots,n-1\};\label{e:taui}\\
&\eta:=(\eta_1,\ldots,\eta_n)\in X\quad\text{where}\quad
	\left\{\begin{aligned}
	&\eta_1:=\tau_1;\quad
	\eta_n:=\tau_{n-1};\quad\text{and}\\
	&\eta_i:=\tau_{i-1}+\tau_{i}\quad\text{for}\quad
	i\in\{2,\ldots,n-1\}.
	\end{aligned}\right.\label{e:etai}
\end{align}
\end{subequations}

\subsection{Area between two linear splines}
\label{ss:Asps}
\begin{figure}[H]
\centering
\includegraphics{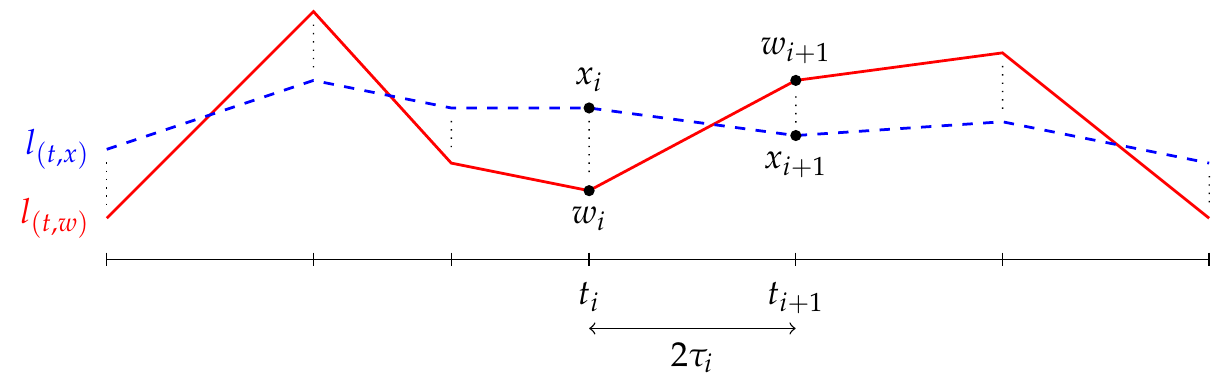}
\caption{Area between two linear splines}\label{f:1}
\end{figure}
Using Section~\ref{ss:geo-area} and Section~\ref{ss:challenging}, 
we estimate the area between the two line segments 
$[(t_i,x_i),(t_{i+1},x_{i+1})]$ and $[(t_i,w_i),(t_{i+1},w_{i+1})]$
(see Figure~\ref{f:1}) by 
\begin{equation}\label{e:0617d}
A_i(x_i,x_{i+1})=\tau_i\cdot f(x_i-w_i,x_{i+1}-w_{i+1}),
\end{equation}
where the value of $A_i$ depends on the norm $f$ as shown in the
following table:

\begin{center}
\begin{tabular}{l|c}
Norm $f$ & Value of $A_i(x_i,x_{i+1})$ \\[+2mm]
\hline
$\ell = \|\cdot\|_1$ &  upper estimate of the area\\[+2mm]
\hline
hexagonal stadium & upper estimate of the area\\[+2mm] 
\hline
stadium & exact area
\end{tabular}
\end{center}

Then the total (absolute) area between two linear 
splines $l_{(t,x)}$ and $l_{(t,w)}$ is estimated by 
\begin{equation}\label{e:0617b}
A(x)=\sum_{i=1}^{n-1}A_i(x_i,x_{i+1})=\sum_{i=1}^{n-1}
\tau_i\cdot f(x_i-w_i,x_{i+1}-w_{i+1}).
\end{equation}

Next, we will compute the proximity operators for the area estimate
$A(x)$. While we are able to explicitly compute the proximity operators for
each term of $A$ (see Theorem~\ref{t:pla-prx}), the overall sum
$A$ does not appear to admit a simple formula.
To deal with $A(x)$, we split it into two parts, 
\begin{subequations}\label{e:0616a}
\begin{equation}
\begin{aligned}
A_{\rm odd}(x)
&=\tau_1\cdot f(x_1-w_1,x_2-w_2)
+\tau_3\cdot f(x_3-w_3,x_4-w_4)+\cdots\\
&=\sum_{i\in\{1,\ldots,n-1\}\cap (1+2\NN)}
\tau_i\cdot f(x_i-w_i,x_{i+1}-w_{i+1}),
\end{aligned}
\end{equation}
and
\begin{equation}
\begin{aligned}
A_{\rm even}(x)
&=\tau_2\cdot f(x_2-w_2,x_3-w_3)
+\tau_4\cdot f(x_4-w_4,x_5-w_5)+\cdots\\
&=\sum_{i\in\{1,\ldots,n-1\}\cap (2\NN)}
\tau_i\cdot f(x_i-w_i,x_{i+1}-w_{i+1}),
\end{aligned}
\end{equation}
\end{subequations}
so that 
\begin{equation}\label{e:Ax}
A=A_{\rm odd}+A_{\rm even}. 
\end{equation}
As the functions in \eqref{e:0616a} are decoupled into independent
pairs of real variables, the proximity operators can be computed
in parallel.
Thus, grouping
\begin{equation}
X\ni(y_1,\ldots,y_n)= 
\Big((y_1,y_2),(y_3,y_4),\cdots\Big)=
\Big(y_1,(y_2,y_3),(y_4,y_5),\cdots\Big),
\end{equation}
and using Theorem~\ref{t:pla-prx}, we obtain the following
result:

\begin{theorem}[proximity operators for area estimations]
\label{t:0616b}
Let $A_i$ be given by \eqref{e:0617d} for every
$i\in\{1,\ldots,n-1\}$,
where $f$ is as in the table below. 
Let $A_{\rm odd}$ and $A_{\rm even}$ be defined by
\eqref{e:0616a}, 
let $\gamma\in\RPP$, let $\alpha\in\RPP$, and let $x\in X$.
Then the proximity operators
of $A_{\rm odd}$ and $A_{\rm even}$ are
\begin{subequations}\label{e:0616b}
\begin{align}\label{e:0616b1}
\prox_{\gamma(\alpha A_{\rm odd})}(x)=\big(\prox_{\gamma(\alpha A_1)}(x_1,x_2),\prox_{\gamma(\alpha A_3)}(x_3,x_4),\ldots\big),
\end{align}
where the last entry in \eqref{e:0616b1} is $x_n$ if $n$ is odd; 
\begin{align}\label{e:0616b3}
\prox_{\gamma(\alpha A_{\rm odd})^*}(x)=\big(\prox_{\gamma(\alpha A_1)^*}(x_1,x_2),\prox_{\gamma (\alpha A_3)^*}(x_3,x_4),\ldots\big),
\end{align}
where the last entry in \eqref{e:0616b3} is $0$ if $n$ is odd;
\begin{equation}\label{e:0616b2}
\prox_{\gamma(\alpha A_{\rm even})}(x)=\big(x_1,\prox_{\gamma(\alpha A_2)}(x_2,x_3),\prox_{\gamma(\alpha A_4)}(x_4,x_5),\ldots\big),
\end{equation}
where the last entry in \eqref{e:0616b2} is $x_n$ if $n$ is even;
\begin{equation}\label{e:0616b4}
\prox_{\gamma(\alpha A_{\rm even})^*}(x)=\big(0,\prox_{\gamma(\alpha A_2)^*}(x_2,x_3),\prox_{\gamma(\alpha A_4)^*}(x_4,x_5),\ldots\big),
\end{equation}
where the last entry in \eqref{e:0616b4} is $0$ if $n$ is even. 
In these formulas, 
\end{subequations}
\begin{subequations}
\begin{align}
&\prox_{\gamma(\alpha A_i)}(x_i,x_{i+1})=(x_i,x_{i+1})
-\gamma\alpha\tau_i \prox_{\uball_*}
(\tfrac{x_i- w_i}{\gamma\alpha\tau_i},
\tfrac{x_{i+1}- w_{i+1}}{\gamma\alpha\tau_i});\\
\text{and}\quad
&\prox_{\gamma(\alpha A_i)^*}(x_i,x_{i+1})=\tau_i \prox_{\uball_*}(\tfrac{x_i-\gamma w_i}{\alpha\tau_i},
\tfrac{x_{i+1}-\gamma w_{i+1}}{\alpha\tau_i}),
\end{align}
\end{subequations}
where $\uball_*$ is the dual unit ball of the norm $f$. 

\begin{tabular}{l|l|l}
{\rm Norm} $f$ & {\rm Formula for $f(z_1,z_2)$}
&{\rm Formula for $\prox_{\uball_*}$} \\[+2mm]
\hline
{\rm $\ell=\|\cdot\|_1$} & $|z_1|+|z_2|$.
& {\rm see \eqref{e:maxproj}} \\[+2mm]
\hline
{\rm hexagonal stadium}
& $\max\big\{|z_1|,|z_2|,|z_1+z_2|\big\}$
& {\rm see \eqref{e:50s1} or \eqref{e:50s3}} \\[+2mm] 
\hline
 & & \\[-4mm]
{\rm stadium}
& $\displaystyle \frac{z_1^2+z_2^2+2\max\{0,z_1z_2\}}{|z_1|+|z_2|}$
& {\rm see \eqref{e:stadiumdualproj}}
\end{tabular}

\end{theorem}

It turns out that if $f=\ell=\|\cdot\|_1$ 
is used for the estimate $A(x)$, then the proximity operators become
simpler since all variables $x_i$ appear separately:

\begin{theorem}[proximity operators for $\ell=\|\cdot\|_1$ area estimation]
Let $l_{(t,x)}$ and $l_{(t,w)}$ be linear splines (see
\eqref{e:l(t,x)}), let
\begin{equation}
A(x)=\sum_{i=1}^{n-1}A_i(x_i,x_{i+1})
=\sum_{i=1}^{n-1}
\tau_i\cdot (|x_i-w_i|+|x_{i+1}-w_{i+1}|)
=\sum_{i=1}^{n}
\eta_i|x_i-w_i|
\end{equation}
be the $\ell=\|\cdot\|_1$ estimation of the area between them
(see \eqref{e:tau-eta}),
let $\gamma\in\RPP$, and let $\alpha\in\RPP$.
Then 
\begin{subequations}
\begin{equation}
\big(\prox_{\gamma(\alpha A)}(x)\big)_i
	= \begin{cases}
	x_i+\gamma(\alpha\eta_i)\displaystyle \frac{w_i-x_i}{|w_i-x_i|}, &\text{if
	$|w_i-x_i|>\gamma\alpha\eta_i$;}\\
	w_i, &\text{otherwise,}
	\end{cases}
\end{equation}	
and
\begin{equation}
	\big(\prox_{\gamma(\alpha A)^*}(x)\big)_i
	= \begin{cases}
	(\alpha\eta_i) \displaystyle \frac{x_i-\gamma w_i}{|x_i-\gamma w_i|}, &\text{if
	$|x_i-\gamma w_i|>\alpha\eta_i$;}\\
	x_i-\gamma w_i, &\text{otherwise.}
	\end{cases}
\end{equation}
\end{subequations}
\end{theorem}

\subsection{Signed area between two linear splines}
\label{ss:Ssps}

Taking into account the signed area between two line segments (see
Section~\ref{ss:signed} and Figure~\ref{f:1b}), 
we obtain the following function of $x$ for the
signed area between two linear splines $l_{(t,x)}$ and
$l_{(t,w)}$: 
\begin{equation}\label{e:Sx}
S\colon X\to\RR\colon 
x\mapsto \sum_{i=1}^{n-1}\tau_i\big((x_i-w_i)+(x_{i+1}-w_{i+1})\big)
=\sum_{i=1}^{n}\eta_i(x_i-w_i)
=\scal{\eta}{x-w},
\end{equation}
where $\tau_i$ and $\eta$ are given by \eqref{e:tau-eta}.
\begin{figure}[H]
\centering
\includegraphics{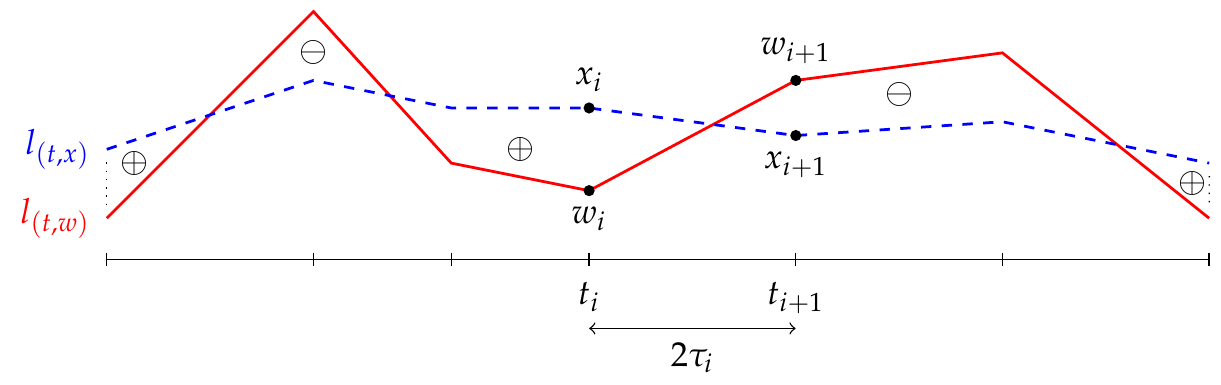}
\caption{Signed area between two linear splines}\label{f:1b}
\end{figure}

Because the signed area function $S$ of \eqref{e:Sx} is simple,
we are able to directly compute the corresponding
proximity operators.
In fact, the following result follows readily 
from Case~5 of Theorem~\ref{t:comm-prxs}:

\begin{theorem}[proximity operators for $|S|$]
Let $l_{(t,x)}$ and $l_{(t,w})$ be two linear splines (see
\eqref{e:l(t,x)}, and let $S$
be given by \eqref{e:Sx}, i.e., the function corresponding to
the signed area between the splines. 
Let $\gamma\in\RPP$ and let $\alpha\in\RPP$.
Then 
\begin{subequations}
\begin{equation}
\prox_{\gamma|\alpha S|}(x)
=x-(\gamma\alpha)\prox_{[-1,1]}
\Big(\tfrac{\scal{\eta}{x-w}}{\gamma\alpha\|\eta\|^2}\Big)\eta
\end{equation}
and
\begin{equation}
\prox_{\gamma|\alpha S|^*}(x)
=\alpha\prox_{[-1,1]}\Big(\tfrac{\scal{\eta}{x-\gamma w}}{\alpha\|\eta\|^2}\Big)\eta.
\end{equation}
\end{subequations}
\end{theorem}

\subsection{Cost functions related to areas in road design problems}

In road design problems, one assumes that the original (vertical) ground
profile is represented by the linear spline
$l_{(t,w)}$ (see \cite{BK13} for details). 
It is required to find a vector $x\in
C_1\cap\cdots C_6$ that is as ``close'' as possible to the vector
$w$. There are several ways to measure this closeness; of
particular interest are the following quantities:

\begin{itemize}
\sepp 
\item 
the \textbf{amount of earth work} (cut and fill) needed.
This amount can be interpreted as the absolute area $A(x)$ between
the two linear splines $l_{(t,x)}$ and $l_{(t,w)}$, which is given
by \eqref{e:Ax} or its polyhedral approximations. 
\item the final \textbf{cut-and-fill balance}. In practice, 
the soil obtained from cutting can be used later for filling.
Therefore, the engineer is also interested in minimizing the final
cut-and-fill balance. This amount is interpreted as the absolute
value of the signed area $S(x)$ (see \eqref{e:Sx}).
\end{itemize}
The measures may be combined by taking conical (i.e., positive
linear) combinations. 
Thus, the problem of interest is to 
\begin{empheq}[box=\mybluebox]{equation}\label{e:probAS}
{\rm Minimize}\quad \alpha A(x)+\beta |S|(x)
\quad\text{subject to}\quad
x\in C_1\cap\cdots\cap C_6,
\end{empheq}
where $A(x)$ is given by \eqref{e:Ax}, $S(x)$ is given by \eqref{e:Sx}, 
and $\alpha$ and $\beta$ are nonnegative weights. 


\section{Douglas--Rachford and Cyclic Intrepid Projections
algorithms}

\label{s:DR}

In this section we briefly review two algorithms we will employ
in numerical experiments. Recall that $X = \RR^N$ and let
$I$ be a nonempty finite set of indices. 

\subsection{Douglas--Rachford Algorithm (DR)}

Consider the problem
\begin{equation}
\label{e:minprob}
\text{minimize}\quad\sum_{i\in I}f_i(x)
\quad\text{subject to}\quad x\in X,
\end{equation}
where each $f_i$ are proper convex lower semicontinuous 
function on $X$. 
The Douglas--Rachford  algorithm, or simply ``DR'' solves 
\eqref{e:minprob} by operating in the product Hilbert space 
\begin{equation}
\bX:=X^I,
\end{equation}
with inner product $\scal{\bx}{\by} := \sum_{i\in I}
\scal{x_i}{y_i}$ for $\bx = (x_i)_{i\in I}$ and $\by =
(y_i)_{i\in I}$. 
Its precise formulation is as follows (see, e.g.,
\cite[Proposition~27.8]{BC}):

Initialize $\bx_0=(x_{0,i})_{i\in I} = (z,\ldots,z)\in\bX$,
where $z\in X$. 
Given $\bx_k\in\bX$, update via 
\begin{subequations}
\begin{align}
\overline{x}_k &:= \frac{1}{|I|}\sum_{i\in I} x_{k,i},\\
(\forall i\in I)\quad y_{k,i} &:= \prox_{\gamma f_i}(2x_{k,i}-\overline{x}_k),\\
(\forall i\in I)\quad x_{k+1,i} &:= x_{k,i} + y_{k,i} -\overline{x}_k,
\end{align}
\end{subequations}
to obtain $\bx_{k+1}$. 
Then the \emph{monitored sequence} $(\overline{x}_k)_{\kkk}$ converges to a 
solution of \eqref{e:minprob}. 

DR finds its roots in the field of differential equations 
\cite{DougRach}. The seminal work by Lions and Mercier
\cite{LioMer} broad to
light the much wider scope of this algorithm. 
Nowadays, there are several variants and numerous studies of 
DR.
We do not describe these variants here because the two modern ones
we experimented with (see \cite{BCH} and
\cite{BAC})\footnote{These variants also require computing
proximity operators of constant multiples of $f_i^*$; see the
previous sections for explicit formulas. We mention also that
these methods allow for great flexibility due to parameters that
can be specified by the user.}
performed similarly to the \emph{plain vanilla} DR. 

\subsection{Method of Cyclic Intrepid Projections (CycIP)}

To describe the method of cyclic intrepid projections,
which has its roots in \cite{Herman}, 
we first need to develop the notion of an intrepid projector.
Suppose that $Z$ is a nonempty closed convex subset of $X$ and
let $\beta\in\RPP$. 
Set $C := \menge{x\in X}{d_Z(x)\leq\beta} $. 
Then the corresponding \emph{intrepid projector} onto $C$ (with
respect to $Z$ and $\beta$) is defined by
\begin{equation}
Q_C\colon X\to X\colon x\mapsto
\begin{cases}
P_Zx, &\text{if $d_Z(x)\geq 2\beta$;}\\
x, &\text{if $d_Z(x)\leq\beta$;}\\
x + \big(\beta-d_Z(x)\big)\displaystyle \frac{x-P_Zx}{\beta}, &\text{otherwise.}
\end{cases}
\end{equation}
Consider the convex feasibility problem
\begin{equation}
\label{e:cfprob}
\text{find}\quad x\in C := \bigcap_{i\in I} C_i\neq\varnothing,
\end{equation}
where each $C_i$ is a nonempty closed convex subset of $X$.
Define $I_0$ by $i\in I_0$ if and only if $i\in I$ and $T_i :=
Q_{C_i}$ is an intrepid projector onto $C_i$; for $i\in I_1 :=
I\smallsetminus I_0$, we set $T_i := P_{C_i}$. 
Given $x_0\in X$, the \emph{method of cyclic intrepid projections
(CycIP)} generates
a sequence $(x_k)_\kkk$ in $X$ via
\begin{equation}
(\forall\kkk)\quad x_{k+1} =  \big(T_{m}T_{m-1}\cdots T_2T_1\big)x_k
\end{equation}
Then the \emph{monitored sequence} $(x_k)_\kkk$ converges to some point in $C$ (see
\cite[Theorem~14]{BIK13}). 

CycIP is just one of many projection methods for solving
\eqref{e:cfprob} (see \cite{SIREV}, \cite{Cegielski},
\cite{CenZen}, \cite{Comb97} and the references therein); 
however, CycIP performed very well in the context 
of road design (see \cite{BIK13} and \cite{BK13}).

\section{Numerical experiments}

\label{s:numerix}

We now return to the optimization problem \eqref{e:probAS}. 
In the context of road design and construction,
$\alpha$ is an averaged unit cost for excavation and embankment,
and $\beta$ is an averaged unit cost for hauling. The values for
$\alpha$ and $\beta$ change with soil types and vary by location;
however, setting $\alpha := 4$ and $\beta :=1$ is a reasonable
assignment based on actual handling cost. 

We will consider Douglas--Rachford algorithm to solve \eqref{e:probAS} with three different estimates of $A(x)$:
\begin{itemize}
\sepp 
\item DRsb: solve problem \eqref{e:probAS} where $A(x)$ is the {\em exact} earth work amount, i.e., using the {\em stadium norm}.

\item DRhb: solve problem \eqref{e:probAS} where $A(x)$ is the upper estimate of earth work amount using the {\em hexagonal stadium norm}.

\item DRlb: solve problem \eqref{e:probAS} where $A(x)$ is the upper estimate of earth work amount using $\ell=\|\cdot\|_1$.
\end{itemize}
Note that at the very least, the engineer must solve the road design feasibility problem 
\begin{equation}\label{e:feas}
{\rm find}\quad x\in C_1\cap\cdots\cap C_6.
\end{equation}
Thus, it is important and interesting to see how much earthwork
one can save by solving the optimization problem \eqref{e:probAS}
rather than the mere feasibility problem \eqref{e:feas}.
Indeed, solving \eqref{e:feas} has been extensively studied in
\cite{BK13}. In particular, the experiments in \cite{BK13} shows
that the method of cyclic intrepid projections (CycIP) is an extremely
fast and efficient algorithm for solving \eqref{e:feas} (for
further information on CycIP see \cite{BIK13}).
Therefore, we will compare the cost-efficiency of DRsb, DRhb, and DRlb to CycIP.

\subsection{Setup and stopping criteria}

Because the Douglas--Rachford algorithm requires the proximity
operators of all function involved, we write \eqref{e:probAS}
as 
\begin{empheq}[box=\mywhitebox]{equation}\label{e:probAAS}
{\rm minimize}\quad \alpha A_{\rm odd}(x)+\alpha A_{\rm even}(x)+\beta |S|(x)+\sum_{i=1}^6 \iota_{C_i}(x)
\quad\text{over}\quad x\in X 
\end{empheq}
in order to use the explicit proximity formulas given in Theorems~\ref{t:comm-prxs} and \ref{t:pla-prx}.

We run the four algorithms described above on $100$ test problems: 6 of which are obtained from real terrain data in British Columbia (Canada), and the rest of which is taken from the test problems in \cite[Section~6]{BK13}.
We set our tolerance at 
\begin{equation}
\ve:=5\cdot 10^{-3}.
\end{equation}
Since {CycIP} is an algorithm aimed at solving the underlying 
feasibility problem, we stop it as soon as a term of the 
monitored sequence
$(x_k)_\kkk$ satisfies\footnote{Recall that the max-norm is given by  $\|x\|_\infty:=\max\{|x_1|,\ldots,|x_n|\}$ for every $x=(x_1,\ldots,x_n)\in \RR^n$.}
\begin{equation}
\max_{i\in\{1,\ldots,6\}}\|x_k-P_{C_i}x_k\|_\infty<\ve. 
\end{equation}
For DRsb, DRhb, DRlb, the Douglas--Rachford-based optimization 
algorithms, we terminate when the first term $\overline{x}_k$ of the monitored sequence $(\overline{x}_k)_\kkk$ satisfies
\begin{equation}
\max_{i\in\{1,\ldots,6\}}\|\overline{x}_k-P_{C_i}\overline{x}_k\|_\infty<\ve
\quad\text{and}\quad
\|\overline{x}_k-\overline{x}_{k-1}\|_\infty<\ve. 
\end{equation}

\subsection{Cost savings}

Although DRsb, DRhb, and DRlb deal with different cost approximations, we are interested in comparing the {\em exact} earthwork cost: recall that given the ground profile $(t,w)$, the exact earthwork amount for a road design $(t,x)$ is
\begin{equation}
F(x):=\alpha A(x)+\beta|S|(x),
\end{equation}
where $A(x)$ is the exact area between two splines $l_{(t,x)}$ and $l_{(t,w)}$, and $S(x)$ is the signed area between these two splines (see Sections~\ref{ss:Asps} and \ref{ss:Ssps}).

For each problem, let $F_{\rm CycIP}$ and $F_{\rm DR}$ be the cost of the road designs obtained by CycIP and DR, respectively. Then the cost saving ratio is given by
\begin{equation}
\Delta_{\rm DR}:=\frac{F_{\rm CycIP}-F_{\rm DR}}{F_{\rm CycIP}}. 
\end{equation}
In the following table, we record the statistics for $\Delta_{\rm DRsb}$, $\Delta_{\rm DRhb}$, and $\Delta_{\rm DRlb}$.
\begin{table}[H]
\centering
\begin{tabular}{|c|c|c|c|c|c|c|c|}
\hline
& Min & 1\textsuperscript{st} Qrt. & Median &
3\textsuperscript{rd} Qrt. & Max & Mean & Std.dev.\\
\hline
$\Delta_{\rm DRsb}$ & $-0.11\%$ & 6.38\% & 12.4\% & 18.82\% & 73.58\% & 14.90\% & 13.91\%\\[+1mm]
\hline
$\Delta_{\rm DRhb}$ & $-0.49\%$ & 6.02\% & 11.96\% & 18.46\% & 72.23\% & 14.56\% & 13.75\%\\[+1mm]
\hline
$\Delta_{\rm DRlb}$ & $-0.12\%$ & 5.30\% & 11.41\% & 17.00\% & 72.73\% & 13.87\% & 13.19\%\\[+1mm]
\hline
\end{tabular}
\caption{Cost savings: DR vs. CycIP (higher is better)}
\end{table}

Theoretically, we expect the cost saving of every optimization algorithm to be nonnegative. However, we observe (small) negative savings by either DR algorithms in $8$ out of $100$ test problems. In fact, because of the $\ve$-tolerance in our stopping criteria, the DR algorithms might stop before attaining optimality.

\subsection{Performance profiles}

To compare the performance of the algorithms, 
we use \emph{performance profiles}\footnote{
For further information on performance profiles, 
we refer the reader to \cite{DM}.}:
for every $a\in\mathcal{A}$ and 
for every $p\in\mathcal{P}$, we set
\begin{equation}
r_{a,p} :=
\frac{k_{a,p}}{\min\menge{k_{a',p}}{a'\in\mathcal{A}}} \geq 1,
\end{equation}
where $k_{a,p}\in\{1,2,\ldots,k_{\max}\}$ is the number of iterations that
$a$ requires to solve $p$ and $k_{\max}$ is the maximum number of iterations
allowed for all algorithms. 
If $r_{a,p} = 1$, then $a$ uses the least number of iterations to
solve problem $p$. 
If $r_{a,p} > 1$, then $a$ requires $r_{a,p}$
times more iterations for $p$ than the algorithm that uses the least
number of iterations for $p$.
For each algorithm $a\in\mathcal{A}$, we plot the function
\begin{equation}
\rho_a\colon \mathbb{R}_+\to[0,1]\colon \kappa\mapsto 
\frac{\card\menge{p\in\mathcal{P}}{\log_2(r_{a,p})\leq\kappa}}{\card
\mathcal{P}},
\end{equation}
where ``$\card$'' denotes the cardinality of a set. 
Thus, $\rho_a(\kappa)$ is the percentage of problems that algorithm
$a$ solves within factor $2^\kappa$ of the best algorithms. 
Therefore, an algorithm $a\in\mathcal{A}$ 
is ``fast'' if $\rho_a(\kappa)$ is large for $\kappa$ small;
and $a$ is ``robust'' if $\rho_a(\kappa)$ is large for $\kappa$ large. 

The following figure shows the performance profiles for the three DR algorithms.

\begin{figure}[H]
\centering
\includegraphics[width=\textwidth]{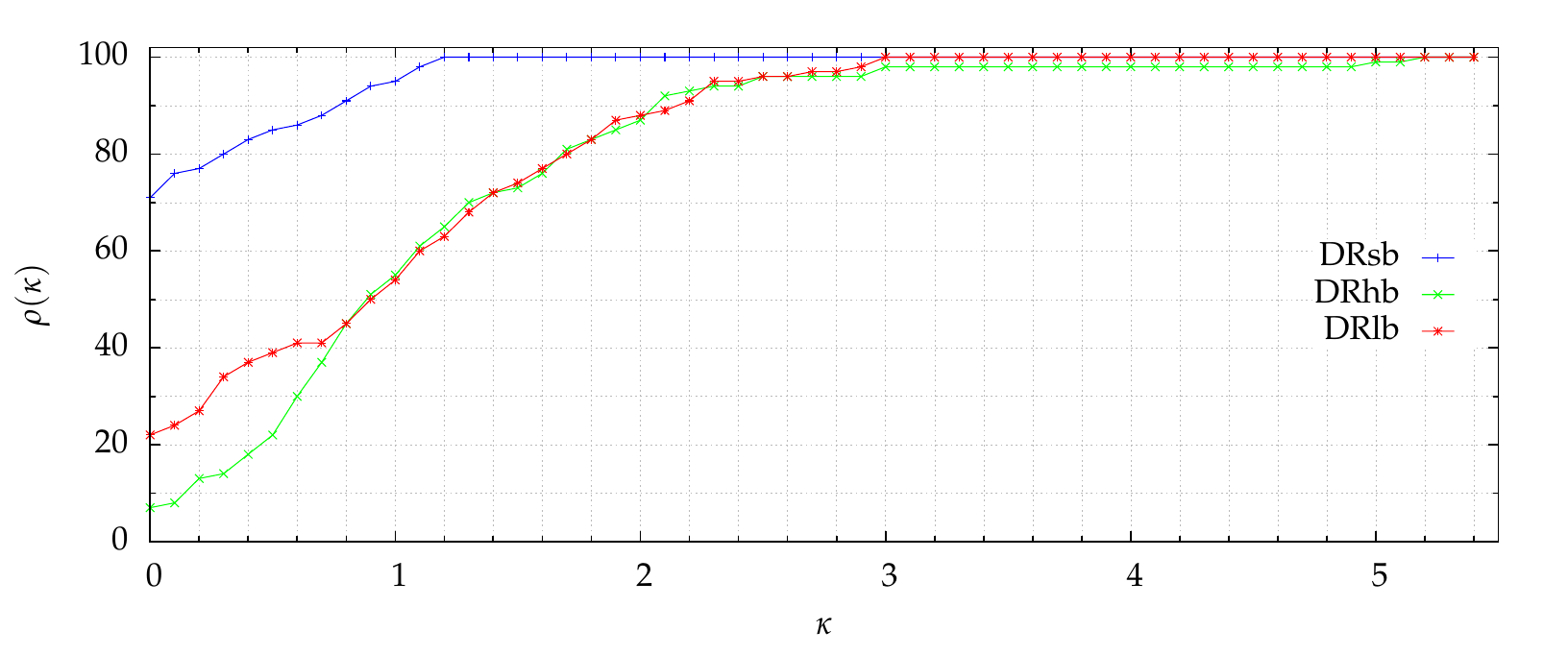}
\caption{Performance profiles (by number of iterations)}
\end{figure}

Note that, the performance profiles only reflect the number of
iterations needed, but they do not take into account the complexity
of proximity operator computations.

\subsection{Problems with real terrain data of BC}

In this section, we present the statistics for the 6 problems that
use real terrain data of British Columbia (Canada).
The problems represent 6 different design alternatives for a 
(hypothetical) high-speed bypass of the city of Merritt, which
would connect Highway 97C directly with the Coquihalla Highway.
The bypass starts at the intersection of the Okanagan Connector Hwy 97C with
the Princeton-Kamloops Hwy 5A, and follows westwards, joining the Coquihalla Hwy 5
near the Kane Valley and Coldwater Rd intersection.

As an example, one of the problems is to build a highway alternative that is
$27.805$ kilometer long and $10.4$ meter wide 
with a design speed of $110$ km/h and a maximum slope of $5\%$.
Starting from the original ground profile (the brown curve in Figure~\ref{fig:init}),
we select the points $(t_i,w_i)_{i\in\{1,\ldots,n\}}$ and create the
initial road design $l_{(t,w)}$ (which is the linear spline generated
by the chosen points).

\begin{figure}[H]
\centering
\includegraphics[width=\textwidth]{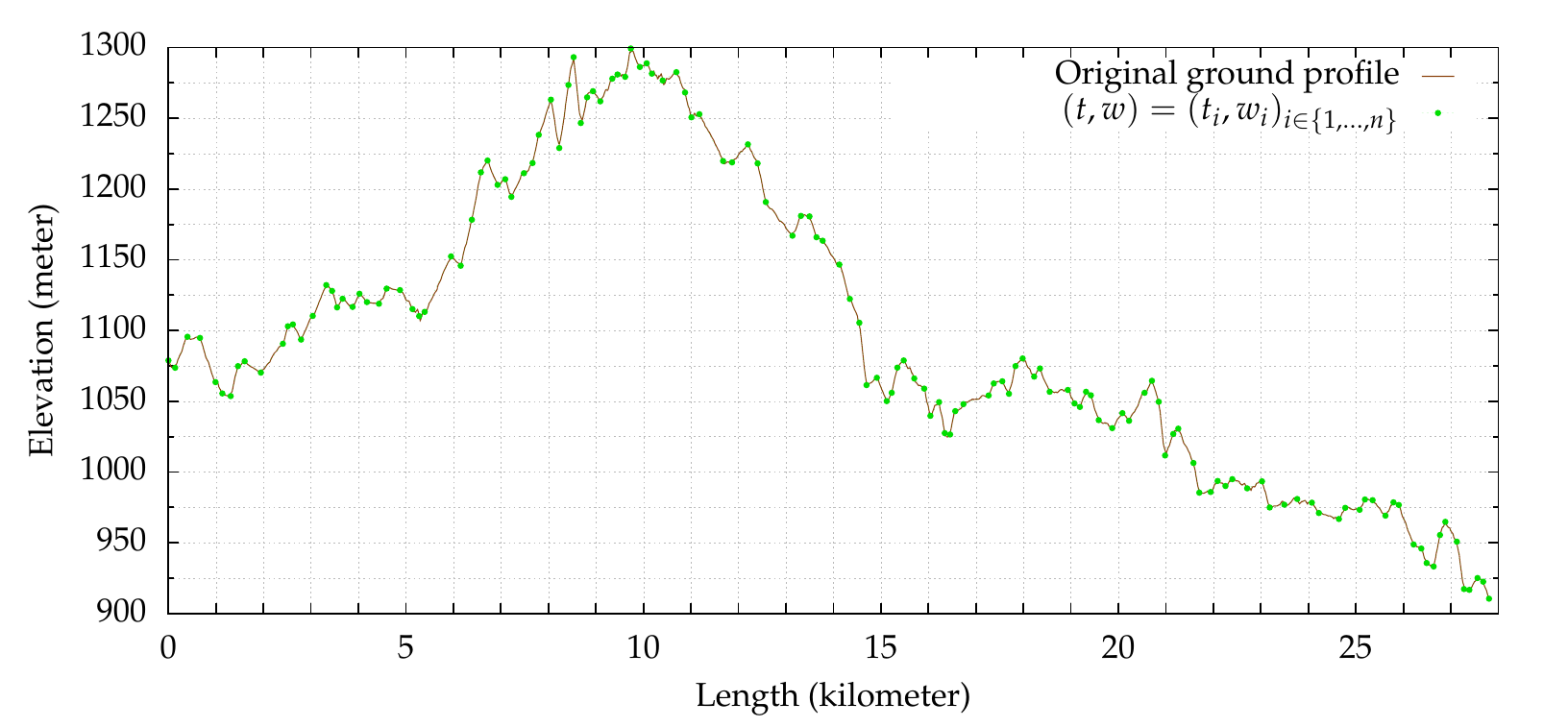}
\caption{Initial road design $l_{(t,w)}$ from the original ground
profile.}
\label{fig:init}
\end{figure}

\noindent
This initial design $l_{(t,w)}$ is usually infeasible, and we use
$w$ as the starting point for the algorithms. The following
two figures show the so-obtained road designs.

\begin{figure}[H]
\centering
\includegraphics[width=\textwidth]{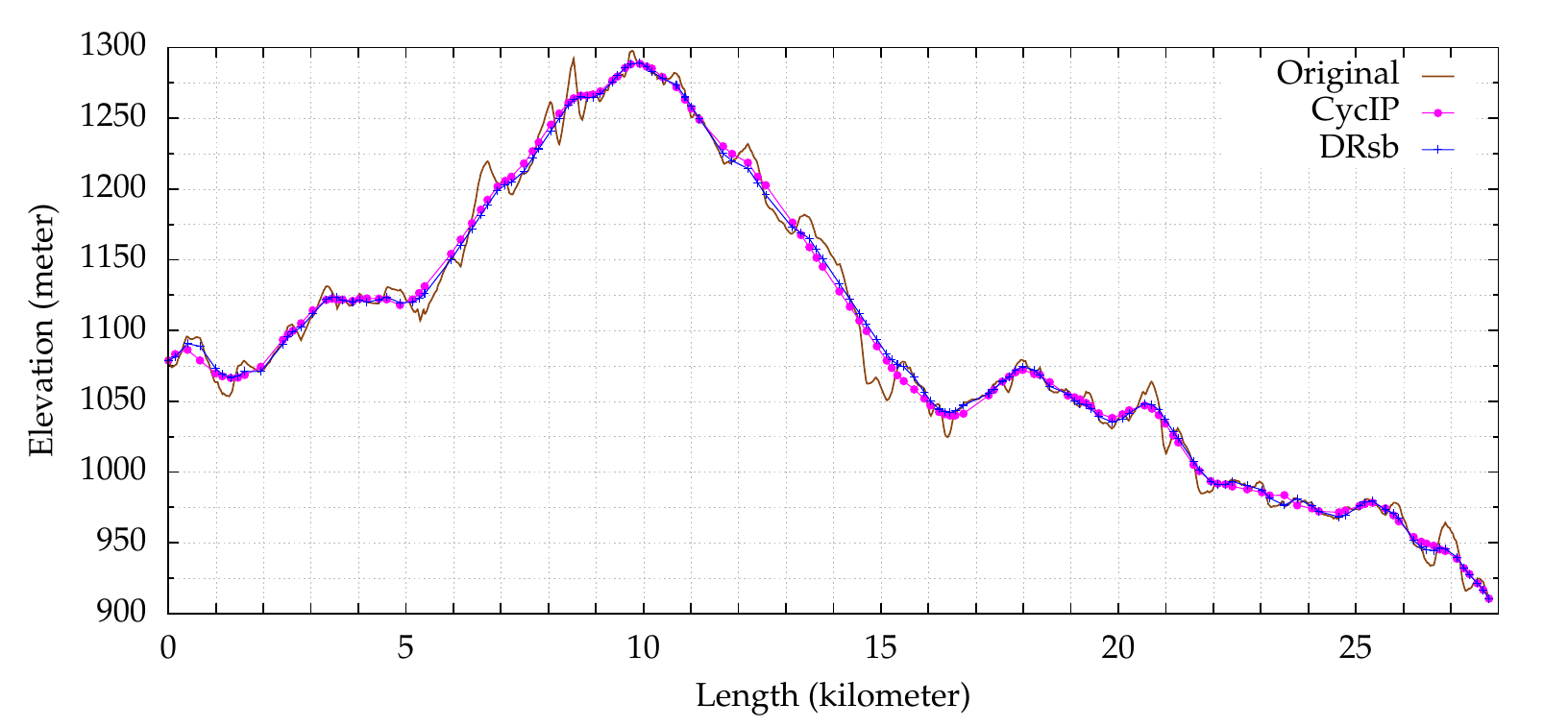}
\caption{Road designs obtained by CycIP and DRsb.}
\end{figure}

\begin{figure}[H]
\centering
\includegraphics[width=\textwidth]{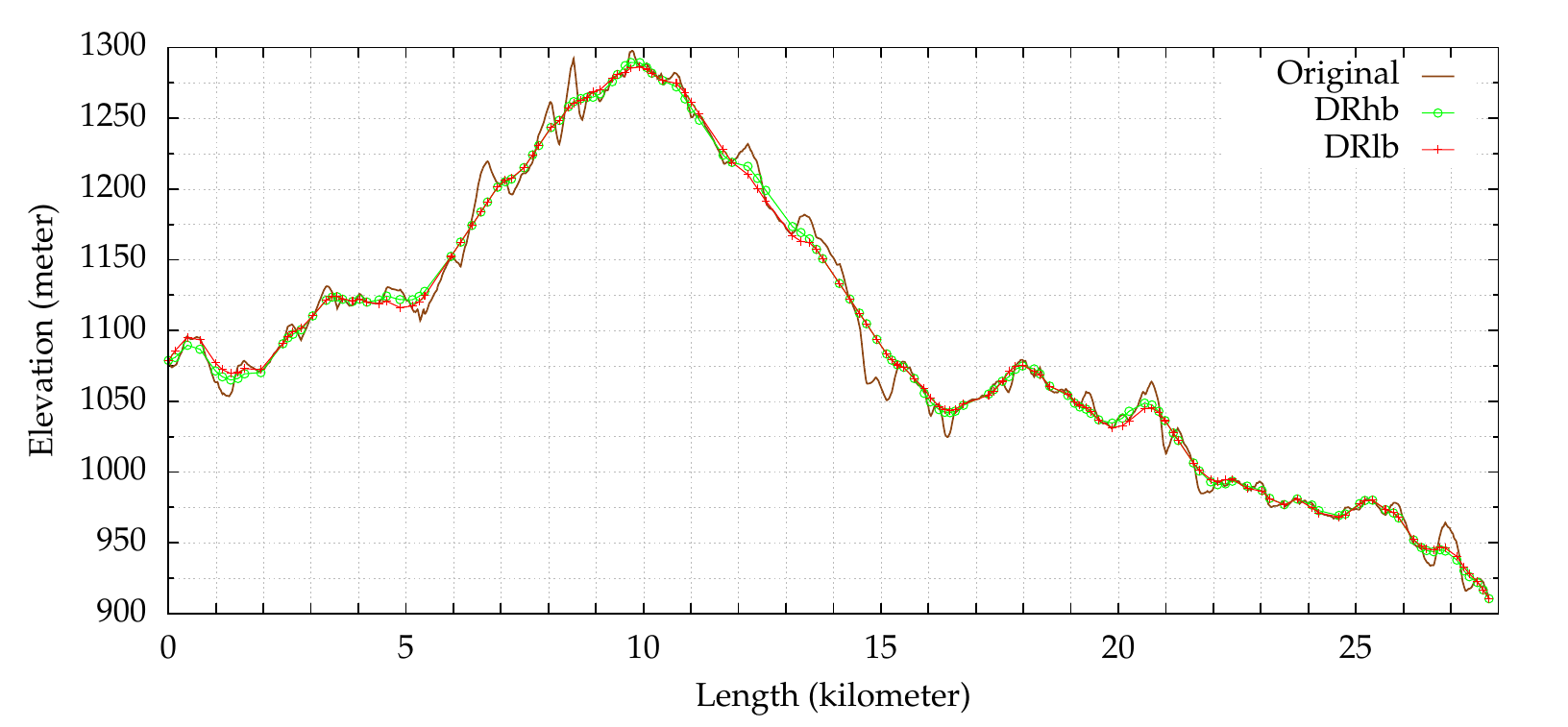}
\caption{Road designs obtained by DRhb and DRlb.}
\end{figure}

\noindent
These road designs are indeed different as seen in the two diagrams below. Figure \ref{fig:mass} presents a \emph{mass diagram}. The mass diagram is a plot of the cut and fill volumes along the road (where cuts are positive and fills are negative). Hence, a mass diagram that finishes closer to zero indicates a better balance between cut and fill. Figure \ref{fig:cummass} shows a cumulative mass diagram, where cut and fills are both taken as positive.

\begin{figure}[H]
\centering
\includegraphics[width=\textwidth]{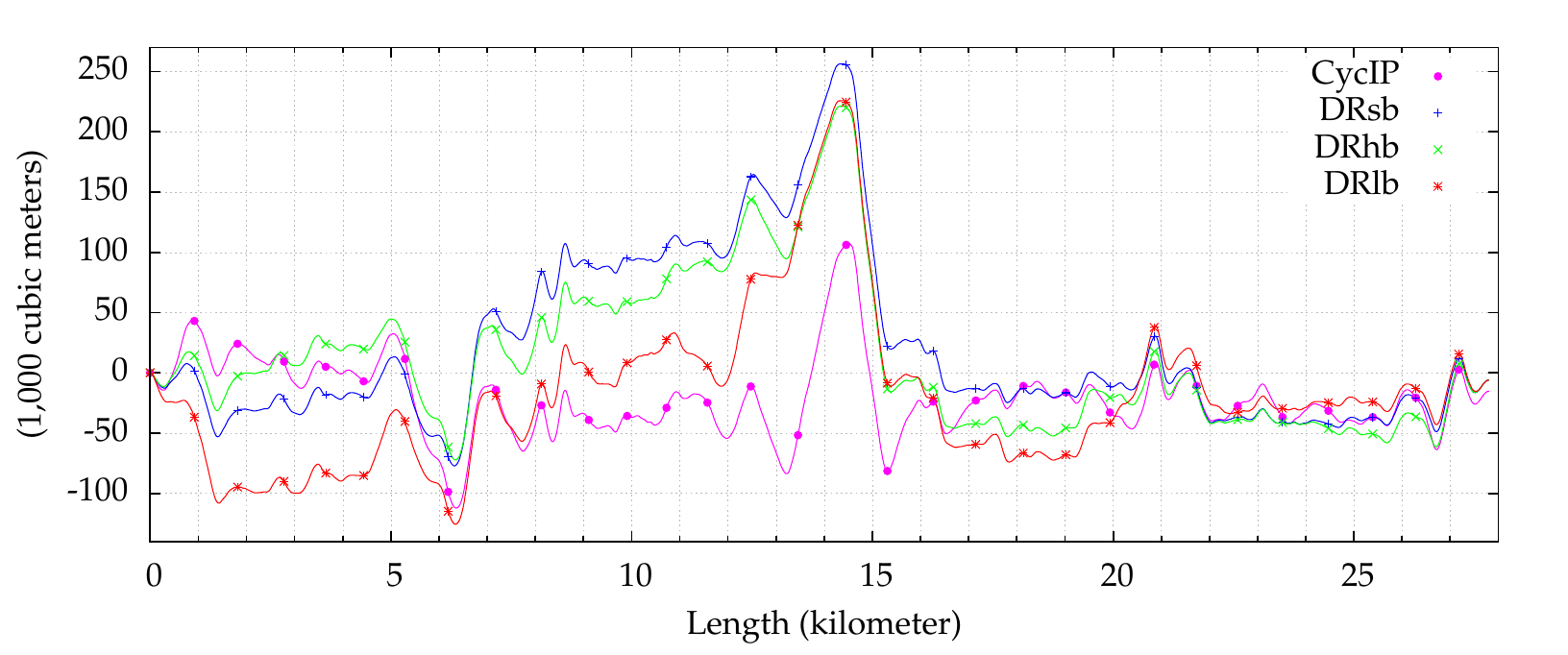}
\caption{Mass diagrams \label{fig:mass}}
\end{figure}

\begin{figure}[H]
\centering
\includegraphics[width=\textwidth]{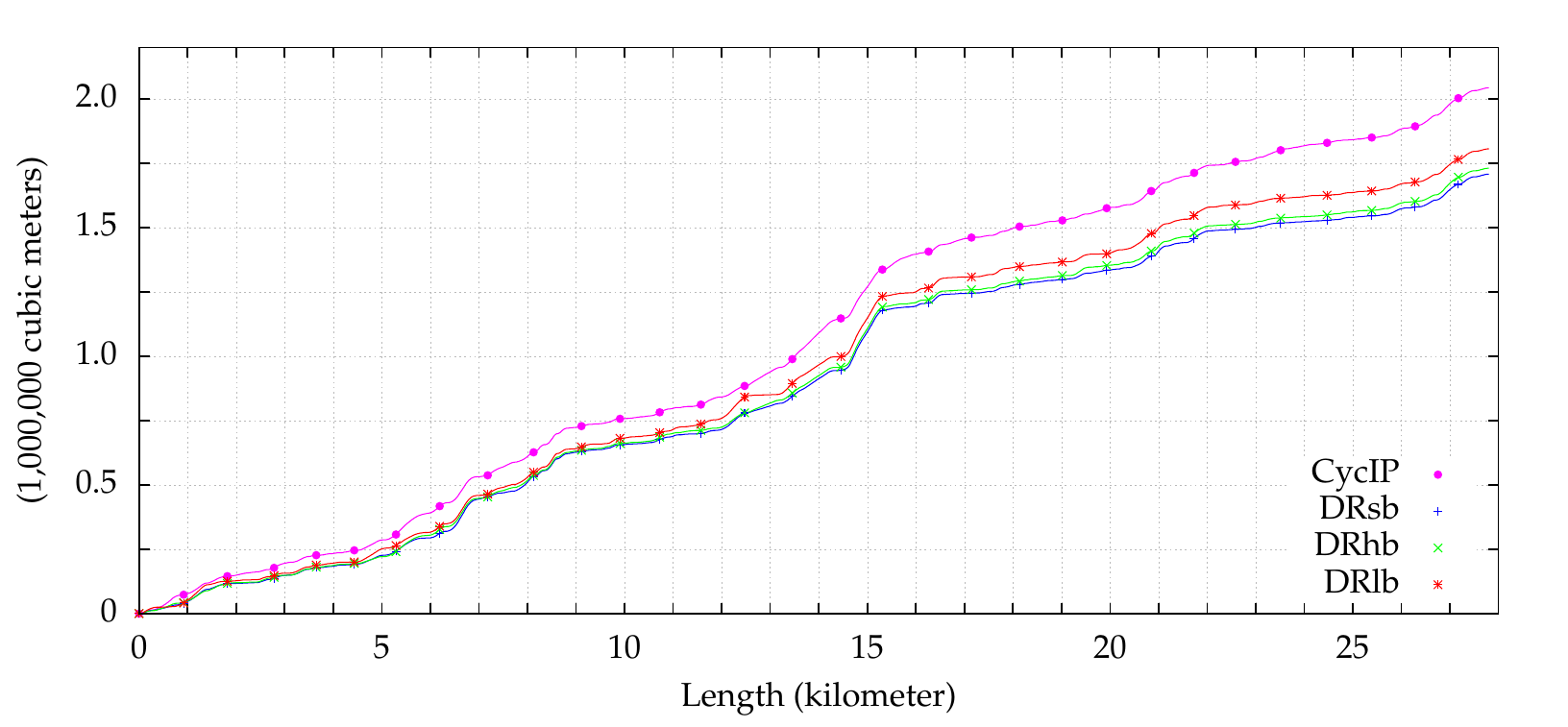}
\caption{Cumulative cut-and-fill amount \label{fig:cummass}}
\end{figure}
\noindent
We set the cost for cut-and-fill at \$5.23 per cubic meter and the
cost for handling the final cut-and-fill balance at \$1.31 per cubic
meter (notice that the ratio of these two costs is approximately
$4\negthinspace:\negthinspace1$). From the obtained data we
then record the cost for each road design in the next table.

\begin{table}[H]
\centering
\begin{tabular}{|c|c|c|c|c|}
\hline
Algorithms & Cut-and-fill (${\rm m}^3$) & Final balance (${\rm m}^3$) & Earthwork cost (\$) & Saving (\%)\\[+1mm]
\hline
CycIP& $2,043,188.4$ & $-15,273.0$ & $10,703,303$ & $0\%$\\[+1mm]
\hline
DRsb& $1,707,709.5$ & $-5,960.7$ & $8,936,992$ & $16.50\%$\\[+1mm]
\hline
DRhb& $1,730,857.5$ & $-5,996.8$ & $9,058,059$ & $15.37\%$\\[+1mm]
\hline
DRlb& $1,805,893.0$ & $-6,036.3$ & $9,450,468$ & $11.71\%$\\[+1mm]
\hline
\end{tabular}
\caption{Earthwork amount and cost saving}
\end{table}

\subsection{Conclusion}

The results suggest the following:
\begin{itemize}
\sepp 
\item Employing the cost function may reduce the construction cost
significantly. In our particular problem, DRsb can save approximately
$1.76$ million dollars ($16.5\%$), while the savings of DRhb and
DRlb are $1.64$ and $1.25$ millions ($15.37\%$ and $11.71\%$),
respectively.
\item Using the exact cost function (i.e., DRsb) may lead 
to a greater saving. 
\item Using the hexagonal approximation (i.e., DRhb) is 
beneficial for programming purpose while
also maintaining a good saving percentage.
\end{itemize}
The data for the other 5 problems listed next also support
our observations.
\begin{table}[H]
\centering
\begin{tabular}{|c|c|c|c|c|c|}
\hline
Algorithms & Prob.~1& Prob.~2& Prob.~3& Prob.~4& Prob.~5\\[+1mm]
\hline
DRsb & $7.05\%$ & $10.35\%$ & $8.88\%$ & $18.96\%$ & $12.81\%$\\[+1mm]
\hline
DRhb & $7.00\%$ & $10.41\%$ & $8.49\%$ & $18.17\%$ & $12.45\%$\\[+1mm]
\hline
DRlb & $6.88\%$ & $6.35\%$ & $7.7\%$ & $16.0\%$ & $10.04\%$\\[+1mm]
\hline
\end{tabular}
\caption{Cost savings over CycIP}
\end{table}

In summary, the experiments support our belief that the road design
optimization problem can be efficiently solved by employing
variants of the Douglas--Rachford algorithm. 
Future work may concentrate on refining the model and on testing
the algorithms on large-scale data using graphics processing
units.

\section*{Acknowledgement}

HHB was partially supported by the Natural Sciences and Engineering Research Council of Canada and by the Canada Research Chair Program. HMP was partially supported by an NSERC accelerator grant of HHB.
The tables and figures in this paper were obtained with the help of \texttt{Julia} (see~\cite{julia}) and \texttt{Gnuplot} (see~\cite{gnuplot}).

\bibliographystyle{plain}

\end{document}